\documentclass{amsart}  
\usepackage{latexsym,amssymb,palatino,eulervm,color,graphicx,amsmath,amsthm}
\usepackage{amsfonts}
\usepackage{amsthm}
\usepackage{amsmath}
\usepackage{amsfonts}
\usepackage{latexsym}
\usepackage{amssymb}
\usepackage{amscd}
\usepackage[latin1]{inputenc}
\usepackage{verbatim}

\newtheorem{definition}{Definition}[section]
\newtheorem{theorem}[definition]{Theorem}
\newtheorem{lemma}[definition]{Lemma}
\newtheorem{corollary}[definition]{Corollary}
\newtheorem{proposition}[definition]{Proposition}
\theoremstyle{definition}

\newcommand\style{\mathcal }          

\newcommand{\B}{\style{B}}
\newcommand{\M}{\style{M}}

\newcommand\A{{\style A}}
\renewcommand{\H}{\style{H}}
\newcommand{\K}{\style K}

\newcommand\osr{{\style R}}
\newcommand\oss{{\style S}}
\newcommand\ost{{\style T}}

\newcommand\omin{\otimes_{\rm min}}
\newcommand\omax{\otimes_{\rm max}}

\newcommand\oc{\otimes_{\rm c}}
\newcommand\coisubset{\subseteq_{\rm coi}}   

\newcommand\cstar{{\rm C}^*}                              
\newcommand\cstare{{\rm C}_{\rm e}^*}              
\newcommand\cstaru{{\rm C}_{\rm u}^*}              
 
\begin{document}

\title[The Operator System of Toeplitz Matrices]{The Operator System of Toeplitz Matrices}

\author[Douglas Farenick]{Douglas Farenick}
\address{Department of Mathematics \& Statistics, University of Regina, Regina, Saskatchewan S4S 0A2, Canada}
\curraddr{}
\email{douglas.farenick@uregina.ca}
\thanks{Supported in part by the NSERC Discovery Grant program}
 
\subjclass[2020]{46L07, 47L05}

\date{}

\dedicatory{}

\begin{abstract}
A recent paper of A.~Connes and W.D.~van Suijlekom \cite{connes-vansuijlekom2020} identifies the 
operator system of $n\times n$ Toeplitz matrices with the dual of the space of all trigonometric polynomials of degree less
than $n$. The present paper examines this identification in somewhat more detail by showing explicitly that the 
Connes--van Suijlekom isomorphism is a unital complete order isomorphism of operator systems. Applications include 
two special results in matrix analysis: (i) that every positive linear map
of the $n\times n$ complex matrices is completely positive when restricted to the operator subsystem of Toeplitz matrices and
(ii) that every linear unital isometry of the $n\times n$ Toeplitz matrices into the algebra of all $n\times n$ complex matrices is 
a unitary similarity transformation. 

An operator systems approach to Toeplitz matrices yields new insights into the positivity of block Toeplitz matrices,
which are viewed herein as elements of tensor product spaces of an arbitrary operator system with the operator
system of $n\times n$ complex Toeplitz matrices. In particular, it is shown that min and max positivity are distinct if the blocks 
themselves are Toeplitz matrices, and that the maximally entangled Toeplitz matrix $\xi_n$ generates an extremal ray
in the cone of all continuous $n\times n$ Toeplitz-matrix valued functions $f$ on the unit circle $S^1$ whose Fourier coefficients $\hat f(k)$
vanish for $|k|\geq n$. Lastly, it is noted that all positive Toeplitz matrices over nuclear C$^*$-algebras are approximately separable.
\end{abstract}

\maketitle
 
\section{Introduction}
Toeplitz operators and matrices are among the most intensively studied and best understood 
of all classes of Hilbert space operators;
in this paper, they are considered from the perspective of the unital selfadjoint linear subspaces they generate. 
These subspaces of matrices and operators are concrete instances of \emph{operator systems}, 
which in the abstract refer to matrix-ordered involutive complex vector spaces possessing an 
Archimedean order unit \cite{choi--effros1977}.

In addition to classical Toeplitz matrices, this paper considers block Toeplitz matrices, 
which are 
matrices $x$ of the form
\begin{equation}\label{tms-intro}
x 
=\left[ \begin{array}{cccccc} 
s_0 & s_{-1} & s_{-2} &  \dots&  s_{-n+2}& s_{-n+1} \\
s_1 & s_0 & s_{-1} & s_{-2}& \dots & s_{-n+2} \\
s_2 & s_1 & s_0 & s_{-1} &\ddots&   \vdots\\
\vdots & \ddots & \ddots & \ddots &\ddots &s_{-2} \\
s_{n-2} &   &  \ddots & \ddots &\ddots & s_{-1} \\
s_{n-1} & s_{n-2} & \dots & s_{2} &s_{1}&s_0
\end{array}
\right],
\end{equation}
for some $s_{-n+1}, \dots,s_{n-1}$ in an operator system $\oss$. The present 
paper considers such matrices $x$ when $\oss$ is an operator system, and
addresses the issue of positivity for these Toeplitz matrices, particularly in the cases 
$\oss=\M_m(\mathbb C)$, the C$^*$-algebra
of $m\times m$ complex matrices, and $\oss=C(S^1)^{(m)}$, the operator system of $m\times m$
complex Toeplitz matrices.

The work in this paper is strongly motivated by recent results of Connes and van Suijlekom
\cite{connes-vansuijlekom2020} which, among other things, 
identify the operator system of Toeplitz matrices with the
dual space of a function system of trigonometric polynomials.
To explain the contributions of the present paper
and set the notation, 
let $C(S^1)$ denote the unital abelian C$^*$-algebra of all continuous functions $f:S^1\rightarrow \mathbb C$,
where $S^1\subset\mathbb C$ is the unit circle. For each $n\in\mathbb N$, let 
$C(S^1)_{(n)}$ denote the vector space of those $f\in C(S^1)$ for which the Fourier coefficients $\hat f(k)$ of
$f$ satisfy $\hat f(k)=0$ for every $k\in\mathbb Z$ such that $|k|\geq n$. Thus, every $f\in C(S^1)_{(n)}$ is given by
\[
f(z)=\sum_{k=-n+1}^{n-1}\alpha_kz^k ,
\]
as a function of $z\in S^1$, 
where each $\alpha_k =\frac{1}{2\pi}\displaystyle\int_0^{2\pi}f(e^{i\theta})e^{-ik\theta}\,d\theta$. 
The vector space $C(S^1)_{(n)}$ is an 
operator system via the matrix ordering that arises from the identification
of $\M_p\left( C(S^1)_{(n)}\right)$, the space of $p\times p$ matrices with entries from $C(S^1)_{(n)}$, with the space
of continuous functions $F:S^1\rightarrow\M_p(\mathbb C)$, and where the Archimedean order unit is the canonical
one (namely, the constant function $\chi_0:S^1\rightarrow \mathbb C$ given by $\chi_0(z)=1$, for $z\in S^1$).

The operator system of all $n\times n$ Toeplitz matrices over $\mathbb C$
is denoted by $C(S^1)^{(n)}$; the identity matrix in $\M_{n}(\mathbb C)$ is the canonical
Archimedean order unit for $C(S^1)^{(n)}$.

As explained in \cite{choi--effros1977}, if
$\osr$ is an operator system and $\osr^d$ denotes its dual space, then $\osr^d$ is a matrix-ordered $*$-vector space.
Specifically, a matrix $\Phi=[\varphi_{ij}]_{i,j=1}^p$ of linear functionals $\varphi_{ij}:\osr\rightarrow\mathbb C$ is considered
positive whenever the linear map $r\mapsto [\varphi_{ij}(r)]_{i,j=1}^p$ is a completely positive linear map of 
$\osr$ into the algebra $\M_p(\mathbb C)$. Furthermore, if
$\phi:\osr\rightarrow\oss$ is a linear map of operator systems and $\phi^d:\oss^d\rightarrow\osr^d$ denotes the
adjoint transformation as linear mapping of matrix-ordered $*$-vector spaces, then 
$\phi$ is positive if and only if $\phi^d$ is positive. Likewise,
$\phi\otimes\mbox{\rm id}_{\M_p(\mathbb C)}$ is positive if and only if 
$\phi^d\otimes\mbox{\rm id}_{\M_p(\mathbb C)}$ is positive, for $p\in\mathbb N$.

If an operator system $\osr$ has finite dimension, then any faithful state $\varphi$ on $\osr$ serves an Archimedean order unit
for the matrix-ordered $*$-vector space $\osr^d$, thereby giving $\osr^d$ the structure of an operator system. 
Because the linear functional $\mathfrak e_{(n)}:C(S^1)_{(n)}\rightarrow\mathbb C$ given by
\[
\mathfrak e_{(n)}[f]=\hat f(0)=\frac{1}{2\pi}\int_0^{2\pi}f(e^{i\theta})\,d\theta
\]
is a faithful state, we shall henceforth designate $\mathfrak e_{(n)}$ as the Archimedean order unit for the operator
system dual of $C(S^1)_{(n)}$.

The category $\mathcal O_1$ has as its objects operator systems, and as its morphisms unital completely positive linear maps.
Therefore, an isomorphism in this category is a unital completely positive linear map $\phi:\osr\rightarrow\oss$ between operator systems 
$\osr$ and $\oss$ such that $\phi$ is a linear bijection and both $\phi$ and $\phi^{-1}$ are completely positive. (The complete positivity
of a linear bijection $\phi$ is not sufficient to imply the complete positivity of its inverse $\phi^{-1}$.) Such a linear isomorphism is called 
a \emph{unital complete order isomorphism} and we denote the existence of such an isomorphism between operator systems 
$\osr$ and $\oss$ with the notation
\[
\osr 
\simeq
\oss.
\]
This notation above has its own ambiguity, as explicit reference to the Archimedean order units of $\osr$ and $\oss$
is not made. In this paper, whenever we are speaking of specific operator systems $\osr$ and $\oss$ that have had
Archimedean order units $e_\osr$ and $e_\oss$ explicitly
designated, then $\osr\simeq \oss$ implies that there is a complete order isomorphism between these
operator systems that sends $e_\osr$ to $e_\oss$.  

The main result of this paper is the following isomorphism theorem in the operator system category; the remaining results
in this paper are derived from this isomorphism.

\begin{theorem}\label{main result} $C(S^1)^{(n)}\simeq \left(C(S^1)_{(n)}\right)^d$  for every $n\in\mathbb N$.
\end{theorem} 

Theorem \ref{main result}, as stated above, 
was proved in \cite[Theorem 4.5]{farenick--paulsen2012} for $n=2$ using an approach rather
different from the approach of the present paper. In that paper, the operator system $C(S^1)_{(2)}$ arises as
$\oss_1$, the operator system generated by a universal unitary operator.
The approach in the present paper is inspired by and based upon an elegant argument of 
Connes and van Suijlekom \cite[Proposition 4.6]{connes-vansuijlekom2020}, which proves that there is a linear unital order
isomorphism $\phi$ between the operator system $C(S^1)^{(n)}$ and 
the operator system dual $\left(C(S^1)_{(n)}\right)^d$. The 
new contribution in Theorem \ref{main result} is the proof that this linear unital order
isomorphism $\phi$ is a complete order isomorphism.

Theorem \ref{main result} has a curious consequence for positive linear maps $\psi:\M_n(\mathbb C)
\rightarrow \M_m(\mathbb C)$. Recall that a positive linear map $\psi:\M_n(\mathbb C)\rightarrow \M_m(\mathbb C)$ is \emph{decomposable}
if it is the sum of completely positive and completely co-positive linear maps; that is, if there are $n\times m$
matrices $a_1,\dots,a_k$ and $b_1,\dots,b_\ell$
such that
\[
\psi(x)\,=\, \sum_{i=1}^k a_i^*x a_i \,+\, \sum_{j=1}^\ell b_j^*x^t b_j,
\]
where $x^t$ denotes the transpose map on $\M_n(\mathbb C)$ (which is positive but not completely positive). 
If $x$ is a Toeplitz matrix, then
there is a unitary $u$ independent of $x$ for which 
$x^t=u^*xu$ (namely, $u=\displaystyle\sum_{i=1}^ne_{i,n-i+1}$, 
where $e_{ij}$ is a matrix unit of $\M_n(\mathbb C)$).
Thus, the restriction of a decomposable positive linear map to the 
Toeplitz operator system $C(S^1)^{(n)}$ is completely positive. However, for all
$n\geq 3$ there exist indecomposable positive linear maps on $\M_n(\mathbb C)$
\cite{tanahashi--tomiyama1988}, which makes the following assertion somewhat unexpected.

\begin{theorem}\label{auto cp} The restriction of every positive linear map 
$\psi:\M_{n}(\mathbb C)\rightarrow \M_m(\mathbb C)$ to the operator
system $C(S^1)^{(n)}$ of Toeplitz matrices is completely positive.
\end{theorem}

Theorems \ref{main result} and \ref{auto cp} also lead to 
a new proof of the following theorem from \cite{farenick--mastnak--popov2016} concerning linear isometries
of operator systems of Toeplitz matrices.

\begin{theorem}\label{iso}
If $\phi:C(S^1)^{(n)}\rightarrow \M_{n}(\mathbb C)$ is a unital linear isometry, then there exists a unitary matrix $v$
such that $\phi(x)=v^*xv$, for every $x\in C(S^1)^{(n)}$.
\end{theorem}

In \S\ref{S:otp}, to
address questions of matrix positivity for Toeplitz matrices $x$ over operator systems $\oss$, as in equation (\ref{tms-intro}),
operator system tensor products will have a role in giving meaning to various notions of positivity. All of these notions coincide
if the $s_k$ are assumed to be elements of $\M_m(\mathbb C)$ (or, indeed, of any nuclear C$^*$-algebra), but the notions become
distinct if the elements $s_k$ in the Toeplitz matrix $x$ are themselves Toeplitz matrices (Corollary \ref{min neq max}).

Lastly, there is a well-known relationship between Toeplitz operators
on the Hardy space of $S^1$ and essentially-bounded functions (symbols) on $S^1$; this relationship 
is addressed in \S\ref{S:to} for the symbol spaces
$C(S^1)_{(n)}$, giving rise to the identification in the operator system category of certain Toeplitz operators
as dual elements to finite Toeplitz matrices.

\section{Preliminaries}

\subsection{Terminology}
Throughout this paper, when referring to a \emph{positive} 
matrix or operator, what is meant is a selfadjoint operator with spectrum contained in the
halfline $[0,\infty)$; thus, ``positive'' is the same as ''positive semidefinite'' in this terminology.

The \emph{Schur-Hadamard product} (or entry-wise product)
of $d\times d $ matrices $a$ and $b$ is denoted by $a\circ b$ and is the matrix in which each $(i,j)$-entry of $a\circ b$ is given by the product of the corresponding
$(i,j)$-entries of $a$ and $b$. The Schur-Hadamard isometry is the linear map $v:\mathbb C^d\rightarrow\mathbb C^d\otimes\mathbb C^d$ that sends the $k$-th canonical
orthonormal basis vector $e_k$ of $\mathbb C^d$ to $e_k\otimes e_k\in \mathbb C^d\otimes\mathbb C^d$; this isometry has the property that $a\circ b=v^*(a\otimes b)v$,
implying that $a\otimes b$ is positive if $a$ and $b$ are positive.

The canonical matrix units for a full matrix algebra $\M_n(\mathbb C)$ shall be denoted, henceforth, by $e_{ij}$,
and the \emph{analytic shift matrix} in $\M_n(\mathbb C)$ is the lower-triangular
Toeplitz matrix $s$ given 
\[
s=\sum_{j=1}^{n-1}e_{j+1,j} = \left[\begin{array}{ccccc} 0&&&&  \\ 1&0&&&  \\ &1&0&&  \\ &&\ddots&\ddots& \\ &&&1&0 \end{array}\right].
\]

The algebra of bounded linear operators acting on a Hilbert space $\H$ is denoted by $\B(\H)$.
An operator $x\in\B(\H)$ is \emph{irreducible} if the commutant of $\{x,x^*\}$
is trivial (i.e., consists only of
scalar multiples of the identity operator). Likewise, an
operator subsystem $\oss\subseteq\B(\H)$ is irreducible if $\oss'$ (the commutant of $\oss$) is trivial.
Thus, the analytic Toeplitz matrix $s$ is irreducible, along with any operator subsystem of $\M_n(\mathbb C)$ that contains $s$.

\subsection{Operator systems}

Formally, an operator system is a triple $(\oss, \{\mathcal C_n\}_{n\in\mathbb N}, e_\oss)$ consisting of:
\begin{enumerate}
\item a complex $*$-vector space $\oss$;
\item a family $\{\mathcal C_n\}_{n\in\mathbb N}$ of proper cones $\mathcal C_n\subseteq\M_n(\oss)_{\rm sa}$ of selfadjoint matrices
with the properties that $\mathcal C_n\oplus\mathcal C_m\subseteq\mathcal C_{n+m}$ and $\gamma^*\mathcal C_n\gamma\subseteq\mathcal C_m$
for all $n,m\in\mathbb N$ and all linear transformations $\gamma:\mathbb C^m\rightarrow\mathbb C^n$; and
\item an Archimedean order unit $e_\oss$ for the ordered real vector space $\oss_{\rm sa}$.
\end{enumerate}
In concrete situations, operator systems arise as unital $*$-closed subspaces of unital C$^*$-algebras.

If $\oss$ and $\osr$ are operator systems with Archimedean order units $e_\oss$ and $e_\osr$, then $\oss$
is said to be an \emph{operator subsystem} of $\osr$ if $\oss\subseteq\osr$ and $e_\oss=e_\osr$. 
A unital completely positive linear map $\phi:\oss\rightarrow\ost$ is a \emph{unital complete order embedding}
if $\phi$ is a unital complete order isomorphism between
$\oss$ and the operator subsystem 
$\phi(\oss)$ of $\ost$. 

The embedding theorem of Choi and Effros \cite{choi--effros1977} states that every operator system $\osr$ is unitally 
completely order isomorphic to an
operator subsystem of $\B(\H)$, for some Hilbert space $\H$. Therefore, every operator system $\osr$ is capable of generating
a C$^*$-algebra. Two such C$^*$-algebras of note are the C$^*$-envelope, $\cstare(\osr)$, and the universal C$^*$-algebra, $\cstaru(\osr)$, which
satisfy, respectively, minimal and maximal universal properties \cite{hamana1979b,kirchberg--wassermann1998}.

The following result from the literature 
shows that C$^*$-envelope of an operator system can be nuclear, while the universal C$^*$-algebra
of the same operator system need not be exact.

\begin{theorem}\label{lit} For every $n\geq2$,
\begin{enumerate}
\item $\cstare ( C(S^1)^{(n)} )=\M_{n}(\mathbb C)$,
\item $\cstare ( C(S^1)_{(n)}  )= C(S^1)$, and 
\item neither $\cstaru ( C(S^1)^{(n)}  )$ nor $\cstaru ( C(S^1)_{(n)}  )$ is an exact C$^*$-algebra.
\end{enumerate}
\end{theorem}

\begin{proof} The first two assertions are given by Propositions 4.2 and 4.3 of
\cite{connes-vansuijlekom2020}, while the third assertion is derived from
Proposition 6.3 of \cite{kavruk2014}.
\end{proof}

\begin{definition} An operator subsystem $\osr$ of a unital C$^*$-algebra $\A$ is \emph{hyperrigid in $\A$} if, for every 
representation $\pi:\A\rightarrow\B(\H_\pi)$ of $\A$, the ucp map $\pi_{\vert\osr}:\osr\rightarrow\B(\H_\pi)$ has a unique
extension to a completely positive linear map on $\A$.
\end{definition}

In the definition above, the unique completely positive extension of the restriction $\pi_{\vert\osr}$ of $\pi$ to $\osr$ 
is of course $\pi$ itself.

\begin{proposition}\label{hyperrigid} The operator system $C(S^1)^{(n)}$ is hyperrigid in $M_n(\mathbb C)$, for all $n\geq2$.
\end{proposition}

\begin{proof} 
Write the lower-triangular shift matrix $s\in C(S^1)^{(n)}$ as
$s=(1/2)(u+w)$, where $u$ and $w$ are the unitary Toeplitz matrices
\[
u=e_{1,n}+\sum_{i=2}^{n} e_{i,i-1} \mbox{ and } w=-e_{1,n}+\sum_{i=2}^{n} e_{i,i-1}.
\]
As the algebra generated by $s$ and $s^*$ is $\M_n(\mathbb C)$, this shows that
the unitary Toeplitz matrices generate the algebra $\M_n(\mathbb C)$. The hyperrigidity of
$C(S^1)^{(n)}$ now follows from
\cite[Lemma 3.11]{harris--kim2019}, which states that if the unitary elements of an operator subsystem $\osr$
of a unital C$^*$-algebra $\A$ generate $\A$, then $\osr$ is hyperrigid in $\A$.
\end{proof}

The following lemmas on automatic complete positivity will be useful. 

\begin{lemma} \label{cp1}
If $\osr$ and $\oss$ are operator systems such that the C$^*$-envelope of $\oss$ is abelian, then
every positive linear map $\phi:\osr\rightarrow\oss$ is completely positive.
\end{lemma}

\begin{proof} The unital complete order embedding $\iota_{\rm e}:\oss\rightarrow\cstare(\oss)$ has the property that, for any $X\in\M_p(\oss)$, 
the matrix $\iota_e^{[p]}(X)$ is positive in $\M_p\left(\cstare(\oss)\right)$ if and only if $X$ is positive in $\M_p(\oss)$.
In other words, 
\[
\oss\simeq \iota_{\rm e}(\oss).
\]
As every positive linear map of an operator system into a unital abelian C$^*$-algebra is completely positive \cite[Theorem 3.9]{Paulsen-book},
we deduce that
$\iota_{\rm e}\circ\phi$ is completely positive. Hence, $\phi$ is necessarily completely positive.
\end{proof}

\begin{lemma}\label{abelian dual} If $\osr$ is a finite-dimensional operator system such that $\cstare(\osr^d)$ is abelian, then every positive
linear map $\phi:\osr\rightarrow\B(\H)$ is completely positive.
\end{lemma}

\begin{proof} The range of $\phi$ is a $*$-closed subspace of $\B(\H)$. Let $\ost$ denote the operator subsystem of $\B(\H)$ spanned by
$\phi(\osr)$ and the identity operator on $\H$; thus, $\ost$ is a finite-dimensional operator system and $\phi$ is a positive linear map 
$\osr\rightarrow\ost$. Consider the dual linear map $\phi^d:\ost^d\rightarrow\osr^d$; because $\phi$ is positive, the dual map $\phi^d$ is positive.
Because $\cstare(\osr^d)$ is abelian, Lemma \ref{cp1} implies that $\phi^d$ is completely positive; hence,  $(\phi^d)^d=\phi$ is completely
positive.
\end{proof}

\subsection{Tensor representations}
Suppose that an operator system $\osr$
has a linear basis $\{r_{\ell}\}_{\ell}$. Fix $p\in\mathbb N$ and let $\{e_{ij}\}_{i,j=1}^p$ denote the canonical basis of $\M_p(\mathbb C)$ given by the
standard matrix units. The algebraic tensor product $\osr\otimes\M_p(\mathbb C)$ may be identified with the set of $p\times p$ matrices with entries from $\osr$ or, 
alternatively,
as ``$\osr$ with entries from $\M_p(\mathbb C)$.'' As it is this second viewpoint that is required, a brief explanation of what this phrase means is given below. 

If $x\in\osr\otimes\M_p(\mathbb C)$, then there exists scalars $\alpha_{\ell i j}\in\mathbb C$, finitely many of which are nonzero, such that
\[
x=\sum_\ell\sum_i\sum_j \alpha_{\ell i j} (r_\ell\otimes e_{ij}) = \sum_i\sum_j\left(\sum_\ell \alpha_{\ell i j} r_\ell \right)\otimes e_{ij} = \sum_i\sum_j x_{ij}\otimes e_{ij},
\]
where $x_{ij}\in\osr$ is the element $\sum_\ell \alpha_{\ell i j} r_\ell$. This gives the identification of $\osr\otimes\M_p(\mathbb C)$ as $p\times p$ matrices with entries from $\osr$,
which is the usual point of view in the theory of operator systems and operator spaces. On the other hand, we may express $x\in\osr\otimes\M_p(\mathbb C)$ as
\[
x=\sum_\ell\sum_i\sum_j \alpha_{\ell i j} (r_\ell\otimes e_{ij})=\sum_\ell r_\ell\otimes\left( \sum_i\sum_j \alpha_{\ell i j}e_{ij}\right) = \sum_\ell r_\ell\otimes a_\ell,
\]
where $a_\ell\in\M_p(\mathbb C)$ is the matrix $ \sum_i\sum_j \alpha_{\ell i j}e_{ij}$. This give us the conceptual identification of $\osr\otimes\M_p(\mathbb C)$
as $\osr$ with entries from $\M_p(\mathbb C)$.

Similarly, if $\phi:\osr\rightarrow\oss$ is a linear map of operator systems, then the linear map 
$\phi^{[p]}=\phi\otimes\mbox{\rm id}_{\M_p(\mathbb C)}:\osr\otimes\M_p(\mathbb C)\rightarrow\oss\otimes\M_p(\mathbb C)$ evaluated at $x\in\osr\otimes\M_p(\mathbb C)$
is, in our two views of $x$ above, given by
\[
\phi^{[p]}[x]=
\phi^{[p]}\left[\sum_\ell\sum_i\sum_j \alpha_{\ell i j} (r_\ell\otimes e_{ij}) \right]= \sum_i\sum_j \phi(x_{ij})\otimes e_{ij}= \sum_\ell \phi(r_\ell)\otimes a_\ell.
\]

\subsection{Canonical linear bases}

The canonical linear basis for 
the operator system $C(S^1)^{(n)}$ of Toeplitz matrices is given by powers of the shift and their adjoints. That is,  
if $\{r_{-n+1},\dots,r_0,\dots, r_{n-1}\}$ is the set of matrices
defined  by
\[
r_k\;=\; \left\{
       \begin{array}{lcl}
           s^k   &:\quad&      \mbox{if }k\geq0  \\
            (s^*)^k          &:\quad&       \mbox{if }k < 0
           
      \end{array}
      \right\}
      \,,
\]
then $\{r_{-n+1},\dots,r_0,\dots, r_{n-1}\}$ is a linear basis of $C(S^1)^{(n)}$. The identity matrix 
$r_0$ serves as the Archimedean order unit 
for the operator system $C(S^1)^{(n)}$.

The operator system $C(S^1)_{(n)}$ of trigonometric polynomials of degree less than $n$ has a canonical linear basis consisting of functions
$\chi_k:S^1\rightarrow\mathbb C$ defined by $\chi_k(z)=z^k$, for $k=-n+1,\dots,n-1$. Recall that
$\chi_0$ is the Archimedean order unit for the operator system $C(S^1)_{(n)}$. 

Let
$\{\mathfrak e_k\,|\, -n+1\leq k\leq n-1\}$ denote the
dual basis
of the $\{\chi_k\,|\,-n+1\leq k\leq n-1\}$; thus, for each $k$,
\[
\mathfrak e_k [f] =\hat f(k),
\]
for every $f\in C(S^1)_{(n)}$. The faithful state $\mathfrak e_0$ is the designated
Archimedean order unit of the operator system dual $\left(C(S^1)_{(n)}\right)^d$.

\section{The Connes-van Suijlekom Theorem}

The Connes-van Suijlekom theorem \cite[Proposition 4.6]{connes-vansuijlekom2020} is formulated and proved below in the
context of the operator system category. The proof is modeled on the arguments of Connes and van Suijlekom.
 
\begin{theorem}\label{main result proof} $C(S^1)^{(n)}\simeq \left(C(S^1)_{(n)}\right)^d$  for every $n\in\mathbb N$.
\end{theorem} 

\begin{proof} The case $n=1$ is trivial; therefore, it is assumed that $n\geq2$.

Consider the linear map $\phi:C(S^1)^{(n)}\rightarrow\left(C(S^1)_{(n)}\right)^d$ which takes a Toeplitz matrix
$t=[\tau_{k-\ell }]_{k,\ell=0}^{n-1}\in\M_{n}(\mathbb C)$ to the linear functional $\varphi_t:C(S^1)_{(n)} \rightarrow\mathbb C$ defined by
\begin{equation}\label{lf defn}
\varphi_t(f)=\sum_{k=-n+1}^{n-1}\tau_{-k}a_k,
\end{equation}
where $f(z)=\displaystyle\sum_{k=-n+1}^{n-1}a_kz^k$. Note that $\phi$ is injective (and hence surjective) and that $\phi$ sends the
identity matrix in $C(S^1)^{(n)}$ to the
linear functional $f\mapsto\hat f(0)$, which we have identified as the Archimedean order unit $\mathfrak e_0$
of the operator system dual $\left(C(S^1)_{(n)}\right)^d$. Therefore, $\phi$ is a unital linear isomorphism, and 
it remains to show that $\phi$ and $\phi^{-1}$ are completely positive.

To show that $\phi$ is completely positive, fix $p\in\mathbb N$ and consider the linear isomorphism $\phi^{[p]}=\phi\otimes\mbox{\rm id}_{\M_p(\mathbb C)}$.
If $T\in C(S^1)^{(n)}\otimes \M_p(\mathbb C)$, then there are matrices $\tau_\ell\in\M_p(\mathbb C)$ such that
\[
T=\sum_{\ell=-n+1}^{n-1}r_\ell\otimes \tau_{\ell},
\]
which as a matrix is represented as
\begin{equation}\label{e:bt}
T= 
\left[ \begin{array}{cccccc} 
\tau_0 & \tau_{-1} & \tau_{-2} &  \dots&  \tau_{-n+2}& \tau_{-n+1} \\
\tau_1 & \tau_0 & \tau_{-1} & \tau_{-2}& \dots & \tau_{-n+2} \\
\tau_2 & \tau_1 & \tau_0 & \tau_{-1} &\ddots&   \vdots\\
\vdots & \ddots & \ddots & \ddots &\ddots &\tau_{-2} \\
\tau_{n-2} &   &  \ddots & \ddots &\ddots & \tau_{-1} \\
\tau_{n-1} & \tau_{n-2} & \dots & \tau_{2} &\tau_{1}& \tau_0
\end{array}
\right].
\end{equation}
Similarly, $\phi^{[p]}(T)$ is given by 
\[
\phi^{[p]}(T)=\sum_{\ell=-n+1}^{n-1}\phi(r_\ell)\otimes \tau_{\ell} = \sum_{\ell=-n+1}^{n-1}\mathfrak e_\ell\otimes \tau_{\ell}.
\]

To understand the action of $\phi^{[p]}(T)$ on $C(S^1)_{(n)}\otimes\M_p(\mathbb C)$, it is enough to understand the action of $\mathfrak e_\ell\otimes g$, for some
fixed $g\in\M_p(\mathbb C)$. Writing $\mathfrak e_\ell\otimes g$ as a $p\times p$ matrix of linear functionals on $C(S^1)_{(n)}$, we obtain
\[
\mathfrak e_\ell\otimes g=\left[\begin{array}{ccc} g_{11}\mathfrak e_\ell& \dots & g_{1p}\mathfrak e_\ell \\ \vdots & \ddots & \vdots \\
g_{p1}\mathfrak e_\ell &\dots & g_{pp}\mathfrak e_\ell\end{array} \right] .
\]
Therefore, if $F\in C(S^1)_{(n)}\otimes \M_p(\mathbb C)$ is given by $F(z)=\sum_{k=-n+1}^{n-1}a_kz^k$, for some matrices $a_k\in\M_p(\mathbb C)$, 
and if $a_k^{ij}\in\mathbb C$ denotes the $(i,j)$-entry of $a_k$, then
\[
F(z)=\left[ \begin{array}{ccc} \displaystyle\sum_{k=-n+1}^{n-1}a_k^{11}z^k & \dots & \displaystyle\sum_{k=-n+1}^{n-1}a_k^{1p}z^k \\
\vdots & \ddots & \vdots \\
\displaystyle\sum_{k=-n+1}^{n-1}a_k^{p1}z^k & \dots & \displaystyle\sum_{k=-n+1}^{n-1}a_k^{pp}z^k
\end{array}
\right] 
\]
and
\[
\mathfrak e_\ell\otimes g(F) = \left[ \begin{array}{ccc} g_{11}\mathfrak e_\ell\left(\displaystyle\sum_{k=-n+1}^{n-1}a_k^{11}z^k\right) & \dots & g_{1p}\mathfrak e_\ell\left(\displaystyle\sum_{k=-n+1}^{n-1}a_k^{1p}z^k \right)\\
\vdots & \ddots & \vdots \\
g_{p1}\mathfrak e_\ell\left(\displaystyle\sum_{k=-n+1}^{n-1}a_k^{p1}z^k\right) & \dots & g_{pp}\mathfrak e_\ell\left(\displaystyle\sum_{k=-n+1}^{n-1}a_k^{pp}z^k\right)
\end{array}
\right] ,
\]
implying that
\[
\mathfrak e_\ell\otimes g(F)
=
\left[ \begin{array}{ccc} g_{11}a_\ell^{11} & \dots &  g_{1p}a_\ell^{1p}\\
\vdots & \ddots & \vdots \\
g_{p1}a_\ell^{p1} & \dots &  g_{pp}a_\ell^{pp} \end{array}
\right] = g\circ a_\ell,
\]
the Schur-Hadamard product of $g$ with $a_\ell$. 
Therefore, if $T$ is the block Toeplitz matrix in (\ref{e:bt}), then
the evaluation of $\phi^{[p]}(T)$ at 
$F(z)=\displaystyle\sum_{k=-n+1}^{n-1}a_kz^k$ 
is given by
\[
\phi^{[p]}(T)[F]=\sum_{k=-n+1}^{n-1}\tau_{-k}\circ a_k.
\]

We now show that if $T$ is positive, then $\phi^{[p]}(T)[F]$ is positive for every positive $F\in C(S^1)_{(n)}\otimes\M_p(\mathbb C)$.
To this end, assume that the matrix $T$ in (\ref{e:bt}) is positive. Thus, the matrix $\tilde T$ whose $(k,\ell)$-entry is $\tau_{k-\ell}\otimes 1_p$,
where $1_p$ denotes the identity matrix of $\M_p(\mathbb C)$, is positive in $ \M_n\left(\M_p(\mathbb C)\otimes \M_p(\mathbb C)\right)$.
That is, the matrix $\tilde T$ is a positive operator on the Hilbert space $\H$ constructed from the direct sum of $n$ copies of $\mathbb C^p\otimes\mathbb C^p$.  Hence, if
$\eta\in\H$ is given by $\eta=\displaystyle\bigoplus_{i=0}^{n-1}\eta_i$, where each
$\eta_i\in \mathbb C^p\otimes\mathbb C^p$, then
\begin{equation}\label{e:ip}
\begin{array}{rcl} 
\langle\tilde T\eta, \eta\rangle &=&\displaystyle\sum_{i=0}^{n-1}\displaystyle\sum_{k=0}^{n-1}\langle(\tau_{i-k}\otimes 1_p)\eta_k,\eta_i\rangle \\ && \\
&=&\displaystyle\sum_{\ell=-n+1}^{n-1}\displaystyle\sum_{j\in I_\ell} \langle(\tau_{-\ell}\otimes 1_p)\eta_{\ell+j},\eta_j\rangle \\ && \\
&\geq& 0,
\end{array}
\end{equation} 
where the set $I_j$ is given by
\begin{equation}\label{e:I_j}
I_j = \left\{ j\in\{0,1,\dots,n-1\}\,\vert\,\ell+j\geq0\right\}.
\end{equation}
If $F:S^1\rightarrow \M_p(\mathbb C)$ is a positive matrix-valued function of the form $F(z)=\displaystyle\sum_{k=-n+1}^{n-1}a_kz^k$,
for some $a_k\in\M_p(\mathbb C)$, then, by the operator-valued Riesz-Fej\'er theorem \cite{dritschel2004,Helson-book}, there exist
$b_0,\dots, b_{n-1}\in\M_p(\mathbb C)$ such that $F(z)=H(z)^*H(z)$ for all $z\in S^1$, where
\[
H(z)=\sum_{k=0}^{n-1} b_kz^k.
\]
In computing the product $H(z)^*H(z)$, we obtain 
\begin{equation}\label{e:rf}
F(z)= \displaystyle\sum_{\ell=-n+1}^{n-1}\left(\displaystyle\sum_{j\in I_\ell} b_j^*b_{\ell+j}\right)z^\ell,
\end{equation}
where $I_j$ is the set in (\ref{e:I_j}). Hence, 
\[
a_\ell = \displaystyle\sum_{j\in I_\ell} b_j^*b_{\ell+j}
\]
for each $\ell\in\{-n+1,\dots,n-1\}$.

Consider the element $q=\displaystyle\sum_{\ell=-n+1}^{n-1}\tau_{-\ell}\otimes a_\ell$ in $\M_p(\mathbb C)\otimes \M_p(\mathbb C)$, and select any vector $\xi\in\mathbb C^p\otimes\mathbb C^p$. For each
$i\in\{0,1,\dots,n-1\}$, set $\eta_i=(1_p\otimes b_i)\xi$ and $\eta=\displaystyle\bigoplus_{i=0}^{n-1}\eta_i\in\H$.
Using
\[
\begin{array}{rcl} q&=& \displaystyle\sum_{\ell=-n+1}^{n-1}\tau_{-\ell}\otimes \left(\displaystyle\sum_{j\in I_\ell} b_j^*b_{\ell+j}\right) \\ && \\  
&=&  \displaystyle\sum_{\ell=-n+1}^{n-1}\displaystyle\sum_{j\in I_\ell} (1_p\otimes b_j)^*(\tau_{-\ell}\otimes 1_p)(1_p\otimes b_{\ell+j}),
\end{array}
\]
we deduce that
\begin{equation}\label{e:q}
\begin{array}{rcl}
\langle q\xi,\xi\rangle &=& 
\displaystyle\sum_{\ell=-n+1}^{n-1}\displaystyle\sum_{j\in I_\ell} \left\langle (\tau_{-\ell}\otimes 1_p)(1_p\otimes b_{\ell+j})\xi, (1_p\otimes b_j)\xi\right\rangle \\ && \\
&=& \displaystyle\sum_{\ell=-n+1}^{n-1}\displaystyle\sum_{j\in I_\ell} \langle(\tau_{-\ell}\otimes 1_p)\eta_{\ell+j},\eta_j\rangle \\ && \\
&=& \langle \tilde T\eta,\eta\rangle \\ && \\
&\geq& 0.
\end{array}
\end{equation}
Hence, $q$ is a positive element of $\M_p(\mathbb C)\otimes \M_p(\mathbb C)$. 

Denote by $v$ the Schur-Hadamard
isometry $v:\mathbb C^p\rightarrow\mathbb C^p\otimes \mathbb C^p$ that implements Schur-Hadamard multiplication of matrices $x$ and $y$ via $x\circ y=v^*(x\otimes y)v$.
Then, because $q$ is positive, so is
\[
\sum_{k=-n+1}^{n-1}\tau_{-k}\circ a_k =  \sum_{k=-n+1}^{n-1}v^*(\tau_{-k}\otimes a_k )v = v^*qv,
\]
which proves that $\phi^{[p]}$ is a positive linear map. Hence, $\phi$ is completely positive.

Turning now to the proof that $\phi^{-1}$ is completely positive, we begin by showing $\phi^{-1}$ is positive. To this end, let
$\psi$ be a pure state on $C(S^1)_{(n)}$, and consider weak*-closed convex set $\mathfrak S_\psi$ of all states on $C(S^1)$ that extend $\psi$. By the Krein-Milman theorem
$\mathfrak S_\psi$ has an extreme point $\Psi$, and this extreme point is necessarily an extreme point of the state space of the abelian C$^*$-algebra
$C(S^1)$; hence $\Psi$ is a point evaluation at some point $\lambda\in S^1$, implying that
$\psi(f)=\displaystyle\sum_{k=-n+1}^{n-1}a_k\lambda^k$, for all $f(z)=\displaystyle\sum_{k=-n+1}^{n-1}a_k z^k$.
Let $\Lambda$ be the positive Toeplitz matrix given by 
\[
\Lambda=\left[ 1\quad \lambda\quad \lambda^2 \quad \dots\quad\lambda^{n-1}\right]^* \left[ 1\quad \lambda\quad \lambda^2 \quad \dots\quad\lambda^{n-1}\right],
\]
and note that $\phi(\Lambda)=\psi$. Hence, $\phi^{-1}(\psi)$ is positive and, therefore, so is $\phi^{-1}(\varphi)$, for every linear functional $\varphi$ that is a limit of positive scalar multiples of
convex combinations of pure states on $C(S^1)_{(n)}$. This proves that $\phi^{-1}$ is positive.

Because $\phi^{-1}$ is positive linear map of $\left(C(S^1)_{(n)}\right)^d$ into $C(S^1)^{(n)}$ , 
its adjoint $\left(\phi^{-1}\right)^d:\left(C(S^1)^{(n)}\right)^d\rightarrow C(S^1)_{(n)}$ is also positive. Furthermore, since the operator system $C(S^1)_{(n)}$ is an operator subsystem
of the abelian C$^*$-algebra $C(S^1)$ and because every positive linear map of an operator system into a unital abelian
C$^*$-algebra is completely positive by Lemma \ref{cp1}, the positive linear map $\left(\phi^{-1}\right)^d$ is completely positive. Thus, 
$\phi^{-1}$ is completely positive.

Hence, $\phi$ is a unital complete order isomorphism.
\end{proof}

\begin{corollary}\label{main result dual}
$C(S^1)_{(n)} \simeq \left(C(S^1)^{(n)}\right)^d$, for every $n\in\mathbb N$.
\end{corollary}

\begin{proof} With any finite-dimensional operator system $\osr$, the Archimedean order
unit $e_\osr$ serves as an Archimedean order
unit for the bidual $\osr^{dd}$, implying that the operator system $\osr^{dd}$ is unitally completely
order isomorphic to $\osr$. Hence, in passing to operator system duals,
Theorem \ref{main result proof} yields the conclusion.
\end{proof}

\section{Operator Systems of Toeplitz Operators}\label{S:to}

The canonical orthonormal basis functions for the Hilbert space $L^2(S^1)$ are given by $e_k(z)=(2\pi)^{-1/2}z^k$, for $k\in\mathbb Z$,
while the Hardy space $H^2(S^1)$ is the subspace of $L^2(S^1)$ having orthonormal basis $\{e_k\}_{k\geq0}$. The projection operator on  $L^2(S^1)$ with range $H^2(S^1)$ is denoted by $P$.
The linear map $\pi:C(S^1)\rightarrow \B\left(L^2(S^1)\right)$ given by $\pi(f)=M_f$, the operator of multiplication by $f$ on the Hilbert space $L^2(S^1)$, is an isometric $*$-representation of 
$C(S^1)$ on $L^2(S^1)$, implying that the linear map $\varphi:C(S^1)\rightarrow \B\left(H^2(S^1)\right)$ defined by $\varphi(f)=PM_f{}_{\vert H^2(S^1)}$ is unital and completely positive. The operator
$\varphi(f)$ is denoted by $T_f$, the Toeplitz operator with symbol $f$.
In expressing a Toeplitz operator $T_f$ as a
matrix with respect to the canonical orthonormal basis of $H^2(S^1)$, the result is an infinite Toeplitz matrix whose
$(\ell,j)$-entry of the matrix is given by $\hat f(\ell-j)$, for $\ell,j\in\{0,1,2,\dots\}$.

For each $n\in\mathbb N$, let
\[
\ost_{(n)}= \left\{T_f\,|\, f\in C(S^1)_{(n)}\right\},
\]
which is an operator system  of Toeplitz operators on $H^2(S^1)$.

\begin{proposition}\label{to} If $\varphi: C(S^1)\rightarrow\B\left(H^2(S^1)\right)$ denotes the symbol map
$\varphi(f)=T_f$, for $f\in C(S^1)$, and if $\varphi_n=\varphi_{\vert C(S^1)_{(n)}}$, for every $n\in\mathbb N$,
then $\varphi_n$ is a unital complete 
order isomorphism of $C(S^1)_{(n)}$ and $\ost_{(n)}$.
\end{proposition}

\begin{proof} The function $\varphi_n: C(S^1)_{(n)}\rightarrow\B\left(H^2(S^1)\right)$ is a unital completely
positive linear map with
range $\ost_{(n)}$. In considering Fourier coefficients, the linear map 
$\varphi_n$ is clearly a linear isomorphism, and so we aim to prove that $\varphi_n^{-1}$
is completely positive. 

First note that because $\varphi_n^{-1}:\ost_{(n)}\rightarrow C(S^1)_{(n)}$ and $C(S^1)_{(n)}$ is an operator subsystem
of the abelian C$^*$-algebra $C(S^1)$, the complete positivity of $\varphi_n$ is automatic, by \cite[Theorem 3.9]{Paulsen-book},
once it has been shown that $\varphi_n$ is positive. To this end, let $T_f\in \ost_{(n)}$ be a positive Toeplitz operator
with symbol $f\in C(S^1)_{(n)}$; thus, the essential spectrum $\mbox{\rm Sp}_e(T_f)$ of $T_f$ is a subset of $[0,\infty)$.
Because $\mbox{\rm Sp}_e(T_f)=f(S^1)$ (see, for example, \cite[\S 4.6]{Arveson-book}), the symbol $f$ is a positive
element of $C(S^1)$ and, hence, of $C(S^1)_{(n)}$, thereby proving that 
$\varphi_n$ is a positive map.
\end{proof}
 
In passing to duals and applying Theorem \ref{main result proof}, 
we obtain another result that seems curious upon first encountering it:
the $n\times n$ Toeplitz matrices are dual to the infinite Toeplitz matrices arising from
symbols in $C(S^1)_{(n)}$.

\begin{corollary} $\left(\ost_{(n)}\right)^d \simeq C(S^1)^{(n)}$.
\end{corollary} 

There is another natural ucp map of interest: the one that maps a Toeplitz operator $T_f$ with symbol $f\in C(S^1)_{(n)}$
to the $n\times n$ Toeplitz matrix $t_f=\displaystyle\sum_{k=-n+1}^{n-1}\hat f(k) r_k$. 
(Equivalently, by Proposition \ref{to}, the map that sends the infinite Toeplitz matrix $T_f$
to its $n\times n$ leading principal submatrix $t_f$.)
If we denote this map
by $\psi_n$, then $\psi_n:C(S^1)_{(n)}\rightarrow C(S^1)^{(n)}$ is a ucp bijection; 
however, $\psi_n$ is not a complete order isomorphism.
For example, in the case $n=2$, the function $f(z)=1+z+z^{-1}$ is not positive on $S^1$, even though
$f=\psi_2^{-1}(t_f)$, where the
element $t_f=\left[\begin{array}{rcl} 1& 1 \\ 1& 1\end{array}\right]$ of $C(S^1)^{(2)}$ is positive.

\section{Applications to Matrix Theory}

\begin{theorem}\label{acp proof}
If $\phi:\M_{n}(\mathbb C)\rightarrow \B(\H)$ is a positive linear map, and 
if $\phi_0$ denotes the restriction of $\phi$ to the Toeplitz matrices $C(S^1)^{(n)}$, 
then $\phi_0$ is a completely positive linear map.
\end{theorem}

\begin{proof} The stated assertion is a direct consequence of
 Lemma \ref{abelian dual}, Theorem \ref{main result proof}, and the fact that
the C$^*$-envelope of $\left(C(S^1)^{(n)}\right)^d\simeq C(S^1)_{(n)}$ is abelian.
\end{proof}

To illustrate this result above, consider the indecomposable positive linear  
map $\psi:\M_3(\mathbb C)\rightarrow \M_3(\mathbb C)$
(known as the Choi map; see \cite{tanahashi--tomiyama1988}) given by
\[
\psi\left(
\left[ \begin{array}{ccc} 
\alpha_{11} & \alpha_{12} & \alpha_{13} \\
\alpha_{21} & \alpha_{22} & \alpha_{23} \\
\alpha_{31} & \alpha_{32} & \alpha_{33} \\
\end{array}\right] 
\right) 
=
\left[ \begin{array}{ccc} 
\alpha_{11}+\alpha_{33}    &- \alpha_{12} & -\alpha_{13} \\
-\alpha_{21} & \alpha_{22}+\alpha_{11}&- \alpha_{23} \\
-\alpha_{31} & -\alpha_{32} &  \alpha_{33}+\alpha_{22}\\
\end{array}\right] .
\]
Note that, 
if $x\in \M_3(\mathbb C)$ is a Toeplitz matrix, then $\psi(x)=x\circ g$, the Schur-Hadamard
product of $x$ with the matrix
\[
g= 
\left[ \begin{array}{ccc} 
2 & -1& -1 \\
-1 & 2& -1 \\
-1 & -1 & 2 \\
\end{array}\right] .
\]
By the Ger\v sgorin circle theorem, the eigenvalues of the symmetric matrix $g$ are nonnegative, and so the mapping $\psi$
coincides on $C(S^1)^{(3)}$ with the completely 
positive map on $\M_3(\mathbb C)$ given by the Schur-Hadamard multiplication 
of $y\in\M_3(\mathbb C)$ by the
positive (semidefinite) matrix $g$.

Theorem \ref{acp proof} also admits a version for Toeplitz matrices over nuclear C$^*$-algebras: see Corollary \ref{nuclear2}.

Theorem \ref{main result proof} also leads to an 
alternative proof of the following result first established in \cite{farenick--mastnak--popov2016}  
concerning the structure of unital linear isometries on the operator system of Toeplitz matrices.

\begin{proposition}\label{ti} If $\phi:C(S^1)^{(n)}\rightarrow M_n(\mathcal C)$ is a unital linear isometry, 
then there exists a unitary $v$ such that $\phi(x)=v^*xv$ for every $x\in C(S^1)^{(n)}$.
\end{proposition}

\begin{proof} Let $W(y)$ denote the numerical range of a matrix $y\in \M_n(\mathbb C)$:
\[
W(y)=\bigcap_{\alpha,\beta\in\mathbb C}\left\{\lambda\in\mathbb C\,|\,|\alpha\lambda+\beta| \leq \|\alpha y +\beta 1_n\|\right\}.
\]
Thus, $W\left(\phi(x)\right)=W(x)$ for every $x\in C(S^1)^{(n)}$, 
since $\phi:C(S^1)^{(n)}\rightarrow M_n(\mathcal C)$ is a unital linear isometry.
As this is also the case for the lower-triangular nilpotent shift matrix $s\in C(S^1)^{(n)}$, the numerical ranges of the 
contractions $s$ and $\phi(s)$ coincide;
hence, by a theorem of Wu \cite{wu1998},
there is a unitary $v$ such that
then $\phi(s)=v^*sv$.

Because an element $x\in C(S^1)^{(n)}$ is positive if and only if the numerical range of $x$ is contained in $[0,\infty)$, 
the unital linear isometry $\phi$ preserves positivity. Hence, by Theorem \ref{acp proof}, $\phi$ is a completely positive linear map.
Now consider the unital completely positive map $\psi:\M_n(\mathbb C)\rightarrow \M_n(\mathbb C)$
defined by $\psi(y)=v\phi(y)v^*$. The fixed point set $\mathfrak F^{\psi}=\{y\in\M_n(\mathbb C)\,|\,\psi(y)=y\}$ contains $s$ and $s^*$,
and so $\mathfrak F^{\psi}$ is an irreducible operator subsystem of $\M_n(\mathbb C)$ such that the restriction of $\psi$ to
$\mathfrak F^{\psi}$ is the identity map. This implies, by Arveson's Boundary Theorem \cite{arveson1972,arveson2008,farenick2011b},
$\psi(y)=y$ for all $y\in M_n(\mathbb C)$; in particular, $\phi(x)=v^*xv$, for every $x\in C(S^1)^{(n)}$.
\end{proof}

Let $\A(S^1)^{(n)}$ denote the set of \emph{analytic Toeplitz matrices}, by which is meant
those Toeplitz matrices that are lower-triangular. Note that $\A(S^1)^{(n)}$ is a unital abelian algebra generated by the
shift $s$.

\begin{proposition}\label{ati} If $\phi:\A(S^1)^{(n)}\rightarrow \M_{n}(\mathcal C)$ is a 
unital isometry, then there exists a unitary $v$ such that $\phi(a)=v^*av$ for every $a\in\A(S^1)^{(n)}$.
\end{proposition}

\begin{proof}
Because the operator system  
$C(S^1)^{(n)}=\{a+b^*\,|\,a,b\in\A(S^1)^{(n)}\}$, the mapping $\tilde\phi:C(S^1)^{(n)}\rightarrow \M_{n}(\mathcal C)$
given by 
\[
\tilde\phi(a+b^*)=\phi(a)+\phi(b)^*  
\]
is well-defined and determines a unital completely positive linear map 
\cite[Proposition 2.12, 3.5]{Paulsen-book}. Further, because $\|\alpha\tilde\phi(s)+\beta 1\|=\|\alpha s + \beta 1 \|$ for all
$\alpha,\beta\in\mathbb C$ (by hypothesis), $\tilde\phi(s)$ is a matrix of unit norm with numerical range equal to that of the shift $s$.
Applying the proof of Proposition \ref{ti} to $\tilde\phi$, we obtain the stated structure for $\tilde\phi$ and, hence, for $\phi$.
\end{proof}

\section{Positivity of Block Toeplitz Matrices via Tensor Products}\label{S:otp}

If $\oss$ is any operator system, then with respect to 
the linear basis $\{r_{-n+1},\dots,r_{n-1}\}$ of $C(S^1)^{(n)}$ identified earlier,
an arbitrary element $x$ of
the algebraic tensor product $C(S^1)^{(n)}\otimes\oss$ can be written as
\[
x=\sum_{\ell=\-n+1}^{n-1} r_\ell\otimes s_\ell,
\]
for some $s_{-n+1}, \dots,s_{n-1}\in \oss$. Hence, the 
algebraic tensor product $C(S^1)^{(n)}\otimes\oss$ 
is naturally identified with the $*$-closed complex vector space of
$n\times n$ Toeplitz matrices with entries from $\oss$, whereby the element $x$ above is represented as
\begin{equation}\label{tms}
x=\sum_{\ell=\-n+1}^{n-1} r_\ell\otimes s_\ell
=\left[ \begin{array}{cccccc} 
s_0 & s_{-1} & s_{-2} &  \dots&  s_{-n+2}& s_{-n+1} \\
s_1 & s_0 & s_{-1} & s_{-2}& \dots & s_{-n+2} \\
s_2 & s_1 & s_0 & s_{-1} &\ddots&   \vdots\\
\vdots & \ddots & \ddots & \ddots &\ddots &s_{-2} \\
s_{n-2} &   &  \ddots & \ddots &\ddots & s_{-1} \\
s_{n-1} & s_{n-2} & \dots & s_{2} &s_{1}&s_0
\end{array}
\right],
\end{equation}
Such elements $x\in  C(S^1)^{(n)}\otimes\oss$ 
are, therefore, called \emph{block Toeplitz matrices}. 

The purpose of this section is to consider various ways in which block Toeplitz matrices 
$x$ can be said to be ``positive,'' particularly in the cases where 
$\oss=\M_m(\mathbb C)$ or $\oss=C(S^1)^{(m)}$, as these cases are of special interest in applied mathematics
(for example, \cite{jeuris--vandebril2016}).

With the case $\oss=\M_m(\mathbb C)$, the positivity of block Toeplitz matrices is determined 
by way of the following
theoretical criterion, which is a direct consequence of (the proof of)
Theorem \ref{main result proof}. 

\begin{proposition}\label{pos bt} 
The following statements are equivalent for a matrix $T\in\M_n(\M_m(\mathbb C))$ of the form
\[
T= 
\left[ \begin{array}{cccccc} 
\tau_0 & \tau_{-1} & \tau_{-2} &  \dots&  \tau_{-n+2}& \tau_{-n+1} \\
\tau_1 & \tau_0 & \tau_{-1} & \tau_{-2}& \dots & \tau_{-n+2} \\
\tau_2 & \tau_1 & \tau_0 & \tau_{-1} &\ddots&   \vdots\\
\vdots & \ddots & \ddots & \ddots &\ddots &\tau_{-2} \\
\tau_{n-2} &   &  \ddots & \ddots &\ddots & \tau_{-1} \\
\tau_{n-1} & \tau_{n-2} & \dots & \tau_{2} &\tau_{1}& \tau_0
\end{array}
\right],
\]
where each $\tau_k\in\M_m(\mathbb C)$:
\begin{enumerate}
\item $T$ is positive;
\item for every positive function $F:S^1\rightarrow \M_m(\mathbb C)$ of the form $F(z)=\displaystyle\sum_{k=-n+1}^{n-1}a_kz^k$,
for some matrices $a_k\in\M_m(\mathbb C)$, the matrix
\[
\sum_{k=-n+1}^{n-1}\tau_{-k}\circ a_k
\]
is positive. 
\end{enumerate}
\end{proposition}


\subsection{Operator system tensor products}

One way to approach the positivity question for block Toeplitz matrices for operator systems $\oss$
different from $\M_m(\mathbb C)$
is via tensor products of operator systems.

An \emph{operator system tensor product} \cite{kavruk--paulsen--todorov--tomforde2011}
$\osr\otimes_\sigma\oss$ of operator systems $\oss$ and $\ost$
is an operator system structure $\otimes_\sigma$ on the algebraic tensor product $\osr\otimes\oss$ 
such that:
\begin{enumerate} 
\item $(\osr\otimes\oss, \{\mathcal C_n\}_{n\in\mathbb N}, e_\osr\otimes e_\oss)$ is an operator system, where $\mathcal C_n\subseteq\M_n(\osr\otimes\oss)$, for each $n$,
and $e_\osr$ and $ e_\oss$ denote the Archimedean order units for $\osr$ and $\oss$;
\item $a\otimes b\in \mathcal C_{nm}$, for all $a\in\M_n(\osr)_+$, $b\in\M_m(\oss)_+$, and $n,m\in\mathbb N$;
\item for all $n,m\in\mathbb N$ and all ucp maps $\phi:\osr\rightarrow\M_n(\mathbb C)$ and $\psi:\oss\rightarrow\M_m(\mathbb C)$, the linear map
$\phi\otimes\psi:\osr\otimes_\sigma\oss\rightarrow \M_{nm}(\mathbb C)$ is completely positive.
\end{enumerate}

Suppose that $\osr_1\subseteq\oss_1$ and $\osr_2\subseteq\oss_2$ are inclusions of operator systems. Let $\iota_j:\osr_j\rightarrow\oss_j$ denote the
inclusion maps $\iota_j(x_j)=x_j$ for $x_j\in\oss_j$, $j = 1,2$, so that the map
$\iota_1\otimes\iota_2:\osr_1\otimes\osr_2\rightarrow\oss_1\otimes\oss_2$ is a linear
inclusion of vector spaces. If $\gamma$ and $\sigma$ are operator system structures on $\osr_1\otimes\osr_2$ and $\oss_1\otimes\oss_2$ respectively,
then we use the notation
\[
\osr_1\otimes_\gamma\osr_2\,\subseteq_+\,\oss_1\otimes_\sigma\oss_2
\]
to denote that $\iota_1\otimes\iota_2:\osr_1\otimes_\gamma\osr_2\rightarrow\oss_1\otimes_\sigma\oss_2$ is a
(unital) completely positive map. This notation is motivated by the fact that $\iota_1 \otimes \iota_2$ is a
completely positive map if and only if, for every $p\in\mathbb N$, the cone $\M_p(\osr_1\otimes_\gamma\osr_2)_+$ is
contained in the cone $\M_p(\oss_1\otimes_\sigma \oss_2)_+$.
If, in addition, $\iota_1\otimes\iota_2$ is a complete order isomorphism onto its range, then we write
\[
\osr_1\otimes_\gamma\osr_2\coisubset\oss_1\otimes_\sigma\oss_2\,.
\]
In particular, if $\gamma$ and $\sigma$ are two operator system tensor-product structures on $\osr\otimes\oss$, then
$\osr\otimes_\gamma\oss=\osr\otimes_\sigma\oss$ means that the identity map is a unital complete order isomorphism
(equivalently, that the matrix positivity
cones for $\osr\otimes_\gamma\oss$ and $\osr\otimes_\sigma\oss$ coincide).

\subsection{Minimal and maximal tensor products of Toeplitz matrices}

An operator system tensor product is described by indicating what the matrix positivity cones are.
 
The \emph{minimal tensor product $\omin$} is the familiar spatial tensor product in matrix and operator theory:
if $\osr\subseteq \B(\H)$ and $\oss\subseteq \B(\K)$, where
$\H$ and $\K$ are Hilbert spaces, then $\osr\omin\oss$ is the operator system arising from the
natural inclusion of $\oss\otimes\ost$ into $\B(\H\otimes\K)$. In this regard, Proposition \ref{pos bt} is a characterisation
of the cone
$\left(C(S^1)^{(n)}\omin \M_p(\mathbb C)\right)_+$.

When considering the finite-dimensional operator system $\ost_n$ of Toeplitz operators $T_f$ acting on the Hardy space
$H^2(S^1)$ with symbols $f\in C(S^1)_{(n)}$, the operator system $\ost_n\omin\ost_m$ is the operator system of
two-level Toeplitz operators $T_h$
acting on the Hardy space $H^2(S^1\times S^1)$ of the torus $S^1\times S^1$ using symbols 
from the set $C(S^1\times S^1)_{(n,m)}$
of all continuous functions
$h:S^1\times S^1\rightarrow \mathbb C$ of the form
\[
h(z,w) =\sum_{k=-n+1}^{n-1}\sum_{j=-m+1}^{m-1}\tau_{kj}z^kw^j,
\]
for some $\tau_{kj}\in\mathbb C$.

The \emph{maximal tensor product $\omax$} is the operator system structure on $\osr\otimes\oss$ obtained through
declaring a matrix $x\in\M_p(\osr\otimes\oss)$ to be positive 
if for each $\varepsilon>0$ there are $n,m\in\mathbb N$, $a\in \M_n(\osr)_+$, $b\in\M_m(\oss)_+$, and a linear map 
$\delta:\mathbb C^p\rightarrow\mathbb C^n\otimes\mathbb C^m$
such that
\[
\varepsilon(e_\osr\otimes e_\oss) + x= \delta^*(a\otimes b)\delta.
\]

It is more difficult for a matrix to be ``max positive'' than ``min positive.'' For example, if $x$ is a strictly positive element
of $\oss\omax\osr$, then there exist $N\in\mathbb N$, $A=[a_{ij}]_{i,j}\in \M_N(\oss)_+$, and $B=[b_{ij}]_{i,j}\in \M_N(\osr)_+$
such that $x=\sum_{i,j}a_{ij}\otimes b_{ij}$. 
Therefore, for $\osr\omin\oss=\osr\omax\oss$ to hold, at least one of $\oss$ or $\osr$ ought to be rich in positive elements and matrices.
In general,   
\[
\osr\omax\oss\subseteq_+ \osr\otimes_\sigma\oss \subseteq_+ \osr\omin\oss,
\]
for every operator system tensor product structure $\otimes_\sigma$ on $\osr\otimes\oss$.

The following theorem gives two fundamental results relating the min and max tensor products.

\begin{theorem}\label{tp thms}{\rm (\cite{farenick--paulsen2012,kavruk--paulsen--todorov--tomforde2011})} 
If $\osr$ and $\oss$ are finite-dimensional operator systems and 
$\ost$ is an arbitrary operator system, then
\begin{enumerate}
\item $\ost\omin\M_m(\mathbb C)= \ost\omax\M_m(\mathbb C)$, for all $m\in\mathbb N$, and
\item $\left(\osr\omin\oss\right)^d \simeq \osr^d\omax\oss^d$.
\end{enumerate}
\end{theorem}

Concerning the maximal tensor product of Toeplitz matrices and operators, we have:

\begin{proposition} For all $n,m\in\mathbb N$,
\[
C(S^1)^{(n)}\omax C(S^1)^{(m)} \simeq
\left( C(S^1)_{(n)}\omin C(S^1)_{(m)}\right)^d
\simeq 
\left(\ost_{(n)}\omin\ost_{(m)}\right)^d
\]
and
\[
\ost_{(n)}\omax\ost_{(m)}\simeq (C(S^1)^{(n)})^d\omax (C(S^1)^{(m)})^d
\simeq\left(C(S^1)^{(n)}\omin C(S^1)^{(m)}\right)^d.
\]
\end{proposition}

\begin{proof} This is an immediate consequence of Proposition \ref{to}
and Theorems \ref{main result proof} and \ref{tp thms}.
\end{proof}

Theorem \ref{tp thms} demonstrates that
\[
C(S^1)^{(n)}\omin\M_m(\mathbb C)= C(S^1)^{(n)}\omax\M_m(\mathbb C),
\] 
for all $n,m\in\mathbb N$.
Corollary \ref{min neq max} below shows that if 
we change the entries $s_j$ of the Toeplitz matrix $x$ in (\ref{tms}) from arbitrary
$m\times m$ complex matrices to arbitrary $m\times m$ Toeplitz matrices, then the two forms of matrix
positivity (min and max) are distinct. To prepare for the proof, we require the following lemma.

\begin{lemma} \label{poly roots} Suppose that $p,m\in\mathbb N$ are such 
that $p$ is prime and $p\geq m$. If $\zeta=e^{2\pi i/p}$ and if $f\in C(S^1)_{(m)}$, then
\[
\hat f(0)= \frac{1}{p}\sum_{k=1}^p f(\zeta^k) .
\]
\end{lemma} 

\begin{proof} Write $f$ as $f(z)=\displaystyle\sum_{\ell=-m+1}^{m-1}\alpha_\ell z^\ell$,
where $\alpha_\ell=\hat f(\ell)$ for each $\ell$. Suppose that $\ell\in \{-m+1,\dots,m-1\}$ is
nonzero. Because $p$ is prime, $\zeta^\ell$ is a primitive $p$-th root of unity, and so 
$\{(\zeta^\ell)^k\,|\,k=1,\dots,p\}$ is the set of all $p$-th roots of unity. Thus,
$\displaystyle\sum_{k=1}^p (\zeta^\ell)^k = 0$. Hence,
\[
\displaystyle\sum_{k=1}^p f(\zeta^k) = 
\displaystyle\sum_{k=1}^p\displaystyle\sum_{\ell=-m+1}^{m-1}\alpha_\ell(\zeta^k)^\ell  
= \displaystyle\sum_{\ell=-m+1}^{m-1}\alpha_\ell\left(\displaystyle\sum_{k=1}^p(\zeta^\ell)^k\right)  
=p\alpha_0,
\]
which completes the proof.
\end{proof}

The following theorem was established in the case $n=m=2$ in  
\cite[Theorem 4.7]{farenick--kavruk--paulsen--todorov2014}; the proof below draws from
the proof of that result.

\begin{theorem}\label{min neq max fs}
$C(S^1)_{(n)}\omin C(S^1)_{(m)} \not= C(S^1)_{(n)}\omax C(S^1)_{(m)}$, for all $n,m\geq2$.
\end{theorem}

\begin{proof}
Assume that
$C(S^1)_{(n)}\omin C(S^1)_{(m)} = C(S^1)_{(n)}\omax C(S^1)_{(m)}$, contrary to what we aim to prove.  

For each $j,n\in\mathbb N$ such that $j\leq n$,
let $\iota_{j,n}:C(S^1)_{(j)}\rightarrow C(S^1)_{(n)}$ 
be the canonical inclusion map, and note that this map is completely positive.
Because the tensor products $\omin$ and $\omax$ are functorial, the linear map
$\iota_{j,n}\otimes \iota_{k,m}$, for $j\le n$ and $k\leq m$, 
is a unital completely positive embedding of 
$C(S^1)_{(j)}\omin C(S^1)_{(k)}$ 
into $C(S^1)_{(n)}\omin C(S^1)_{(m)}$ and of $C(S^1)_{(j)}\omax C(S^1)_{(k)}$ into 
$C(S^1)_{(n)}\omax C(S^1)_{(m)}$.

Identify $\M_2(C(S^1)_{(2)}\omin C(S^1)_{(m)})$ with $C(S^1)_{(2)}\omin\M_2(C(S^1)_{(m)})$
and let $x\in C(S^1)_{(2)}\omin\M_2(C(S^1)_{(m)})$ be given by
\[
x=\chi_0\otimes b_0\,+\,
\chi_1\otimes b_1
\,+\, \chi_{-1}\otimes b_{-1},
\]
where $b_0=3\chi_0$, $b_1=\left[\begin{array}{cc} \chi_1& 0 \\ 2\chi_{-1} &-\chi_1 \end{array}\right]$, and
$b_{-1}=b_1^*$. As shown in \cite[Theorem 4.7]{farenick--kavruk--paulsen--todorov2014},
this matrix-valued function is a positive element of 
$C(S^1)_{(2)}\omin\M_2(C(S^1)_{(m)})$ and the eigenvalues of
$x(z,w)$, for $(z,w)\in S^1\times S^1$, are uniformly bounded below by some $\delta>0$.
Therefore, the
matrix $x$ is strictly positive in 
$C(S^1)_{(n)}\omin\M_2(C(S^1)_{(m)})$, which by hypothesis coincides with
$C(S^1)_{(n)}\omax\M_2(C(S^1)_{(m)})$.

As $x\in C(S^1)_{(n)}\omax\M_2(C(S^1)_{(m)})$ is strictly positive, 
there exist, by \cite[Lemma 2.7]{farenick--kavruk--paulsen--todorov2014}, 
$N\in\mathbb N$ and positive matrices
$F=[f_{ij}]_{ij}\in \M_N(C(S^1)_{(n)})$ and $G=[g_{ij}]_{ij}\in \M_N(C(S^1)_{(m)})$
such that
\[
\chi_0\otimes b_0\,+\,
\chi_1\otimes b_1
\,+\, \chi_{-1}\otimes b_{-1} \,=\,\sum_{i,j=1}^N  f_{ij}\otimes g_{ij}.
\]
Express each $f_{ij}$ as $f_{ij}=\displaystyle\sum_{\ell=-n+1}^{n-1}\alpha_\ell^{(ij)}\chi_\ell$
so that the sum above becomes
\[
 \chi_0\otimes b_0\,+\,
\chi_1\otimes b_1
\,+\, \chi_{-1}\otimes b_{-1} \,=\,\sum_{\ell=-n+1}^{n-1}\chi_\ell\otimes 
\left(\sum_{i,j=1}^N   \alpha_\ell^{(ij)} g_{ij}\right).
\]
Thus, by the linear independence of $\{\chi_\ell\,|\,\ell=-n+1,\dots,n-1\}$, 
\[
\sum_{i,j=1}^N   \alpha_\ell^{(ij)} g_{ij} = 0,
\]
whenever $|\ell|>1$, and
\[
b_0 = \sum_{i,j=1}^N   \alpha_0^{(ij)} g_{ij}  \,\mbox{ and } \,
b_1 = \sum_{i,j=1}^N   \alpha_1^{(ij)} g_{ij} = b_{-1}^*.
\]
If $A_0=[\alpha_0^{(ij)} ]_{i,j=1}^N$ and $A_1= [\alpha_1^{(ij)} ]_{i,j=1}^N$, then
$A_0$ is positive, $A_1$ is hermitian, and,
by the proof of \cite[Lemma 4.6]{farenick--kavruk--paulsen--todorov2014}, there exist 
positive complex matrices
$\Gamma_1=[\gamma_{i,j}^{(1)}]_{i,j=1}^N$ and $\Gamma_2=[\gamma_{i,j}^{(2)}]_{i,j=1}^N$
such that $\Gamma_1+\Gamma_2=A_0$ and 
\[
Y=\left[ \begin{array}{cc} h_1 & b_1 \\ b_{-1} & h_2 \end{array}\right]\in \left(\M_4( C(S^1)_{(m)}\right)_+,
\]
where $h_1,h_2\in \M_2(C(S^1)_{(m)}$ are given by 
\[
h_k=\sum_{i,j}\gamma_{ij}^{(k)}g_{ij}, \mbox{ for }k=1,2,
\]
and satisfy $h_1+h_2=b_0$.

As in the proof of \cite[Theorem 4.7]{farenick--kavruk--paulsen--todorov2014}, $Y$ is unitarily equivalent
to a positive
matrix-valued function of the form
\[
\tilde Y (w) = 
\left[ 
\begin{array}{cccc} h_{11}(w) &  h_{12} & 1 & 0 \\
 h_{21}(w) & h_{22}(w) & 2 & -1 \\
1&2& 3-h_{11}(w) &  h_{12}(w) \\
0&-1&  h_{21}(w) & 3-h_{22}(w)
\end{array}
\right],
\]
for some $h_{ij}\in C(S^1)_{(m)}$ in which $\hat h_{12}(0)=\hat h_{21}(0)=0$. By selecting a prime $p\geq m$
and the primitive $p$-th root of unity $\zeta=e^{2\pi i /p}$, we have, by Lemma \ref{poly roots},
\[
\frac{1}{p}\sum_{k=1}^p \tilde Y(\zeta^k)=
\left[ 
\begin{array}{cccc} \hat h_{11}(0) &  0 & 1 & 0 \\
 0& \hat h_{22}(0) & 2 & -1 \\
1&2& 3-\hat h_{11}(0) &  0 \\
0&-1&  0 & 3-\hat h_{22}(0)
\end{array}
\right].
\]
By hypothesis, this matrix above is positive, since $Y$ and $\tilde Y$ are positive. However, 
as shown in the proof of  
\cite[Theorem 4.7]{farenick--kavruk--paulsen--todorov2014}, it is not.
Therefore, the assumption that 
$C(S^1)_{(n)}\omin C(S^1)_{(m)} = C(S^1)_{(n)}\omax C(S^1)_{(m)}$ leads to a contradiction.
\end{proof}

\begin{corollary}\label{min neq max}
$C(S^1)^{(n)}\omin C(S^1)^{(m)} \not= C(S^1)^{(n)}\omax C(S^1)^{(m)}$, for all $n,m\geq2$.
\end{corollary}

\begin{proof} The dual of the equality $C(S^1)^{(n)}\omin C(S^1)^{(m)} = C(S^1)^{(n)}\omax C(S^1)^{(m)}$ is
$C(S^1)_{(n)}\omin C(S^1)_{(m)} = C(S^1)_{(n)}\omax C(S^1)_{(m)}$, which is false by Theorem
\ref{min neq max fs}.
\end{proof}

\subsection{The commuting tensor product}

The 
\emph{commuting tensor product $\oc$} is the operator system structure on $\osr\otimes\oss$ obtained by
declaring a matrix $X\in\M_p(\osr\otimes\oss)$ to be positive if
$(\phi\cdot\psi)^{(p)}(X)$ is a positive operator for all pairs of completely positive maps $\phi :\osr\rightarrow\B(\H)$
and $\psi :\oss\rightarrow\B(\H)$ with commuting ranges, where  
$\phi\cdot\psi$ denotes the linear map
$\osr\otimes\oss\rightarrow\B(\H)$
defined by $\phi\cdot\psi(x\otimes y) = \phi(x)\psi(y)$, for $x\in \osr$, $y\in \oss$.

The relationship of $\oc$ to $\omin$ can be vexing to determine, even for operator systems of low dimension. (The problem of whether $\oss_2\omin\oss_2=\oss_2\oc\oss_2$,
for the operator system $\oss_2$ generated by the unitary generators of the free group C$^*$-algebra $\cstar(\mathbb F_2)$, is equivalent to the
Connes Embedding Problem \cite[Theorem 5.11]{kavruk2014}.) 
In contrast, the relationship of $\oc$ to $\omax$ can often be discerned, as is the case with the Toeplitz operator system and its dual.

\begin{proposition} $C(S^1)^{(n)}\oc C(S^1)_{(n)} \neq C(S^1)^{(n)}\omax C(S^1)_{(n)}$, for all $n\geq2$.
\end{proposition}

\begin{proof} Via the complete order isomorphism $(C(S^1)_{(n)})^d\simeq C(S^1)^{(n)}$, 
the equality of $C(S^1)^{(n)}\oc C(S^1)_{(n)}$ and $C(S^1)^{(n)}\omax C(S^1)_{(n)}$ is possible only
if $C(S^1)_{(n)}$ is completely order isomorphic to a C$^*$-algebra $\A$ \cite[Proposition 4.3]{kavruk2014}.
As the C$^*$-envelope is invariant under complete order isomorphism, this would imply that
$\A=\cstare(\A)\cong\cstare(C(S^1)_{(n)})=C(S^1)$, yielding a linear isomorphism between the 
finite-dimensional vector space $C(S^1)_{(n)}$ and the infinite-dimensional 
vector space $C(S^1)$, which is
impossible.
\end{proof} 

The matrix ordering of $\oss\oc\osr$ is achieved through the canonical embedding of the algebraic tensor
product $\oss\otimes\osr$ into the C$^*$-algebra $\cstaru(\oss)\omax\cstaru(\osr)$ \cite{kavruk--paulsen--todorov--tomforde2011}; that is,
\[
\oss\oc\osr\subseteq_{\rm coi} \cstaru(\oss)\omax\cstaru(\osr).
\]
Even so, the universal C$^*$-algebras of $C(S^1)^{(n)}$ and $C(S^1)_{(n)}$ are not sufficiently tractable (e.g., see Theorem \ref{lit})
to draw additional information from.

\section{The Maximally Entangled Toeplitz Matrix}

\begin{definition}\label{def:se} If $\otimes_\sigma$ is an operator system tensor product structure on $\oss\otimes\osr$, for operator systems $\oss$
and $\osr$, then an element $x\in(\oss\otimes_\sigma\osr)_+$ is:
\begin{enumerate}
\item \emph{$\sigma$-separable}, if there exist $k\in\mathbb N$, $s_1,\dots,s_k\in\oss_+$, and $r_1,\dots,r_k\in\osr_+$ such that
$x=\displaystyle\sum_{j=1}^k s_j\otimes r_j$; or
\item \emph{$\sigma$-entangled}, if $x$ is not $\sigma$-separable.
\end{enumerate}
In the case where $\otimes_\sigma$ is the minimal operator system tensor product $\omin$, then we simply refer to $x\in(\oss\omin\osr)_+$ as being
\emph{separable} or \emph{entangled}.
\end{definition}

A beautiful result of Gurvits shows that positive block Toeplitz matrices with blocks coming from $\M_m(\mathbb C)$ are separable.

\begin{theorem}[Gurvits] Every positive element of $C(S^1)^{(n)}\omin\M_m(\mathbb C)$ is separable.
\end{theorem}

\begin{proof} The proof of Gurvits' theorem given in \cite[\S III]{gurvits--barnum2002} yields the result for the operator system
$C(S^1)^{(n)}\omin\M_m(\mathbb C)$, although the result is predominantly cited in the literature as pertaining to the matrix algebra
$\M_n(\mathbb C)\omin\M_m(\mathbb C)$.
\end{proof}

\begin{corollary}\label{nuclear} If $\A$ is a unital nuclear C$^*$-algebra and $x\in (C(S^1)^{(n)}\omin\A)_+$, then for every $\varepsilon>0$
there exists a separable $y\in (C(S^1)^{(n)}\omin\A)_+$ such that $\|x-y\|<\varepsilon$.
\end{corollary}

\begin{proof} As $\A$ is nuclear, there exists nets $\{\psi_\lambda\}_\lambda$ and $\{\phi_\lambda\}_\lambda$ of
completely positive linear maps $\psi_\lambda:\A\rightarrow\M_{n_\lambda}$ and 
$\phi_\lambda:\M_{n_\lambda}\rightarrow\A$ such that $\lim_\lambda\|\phi_\lambda\circ\psi_\lambda(a)-a\|=0$ for every $a\in\A$.
If $\iota_n$ denotes the identity map on $C(S^1)^{(n)}$, then by Gurvits' theorem there exist, for each $\lambda$,
positive Toeplitz matrices $t_1^\lambda,\dots,t_{k_\lambda}^\lambda\in C(S^1)^{(n)}$ and positive $g_1^\lambda,\dots,g_{k_\lambda}^\lambda\in \M_{n_\lambda}(\mathbb C)$
such that
\[
(\iota_n\otimes\psi_\lambda)[x]=\sum_{j=1}^{k_\lambda} t_j^\lambda\otimes g_j^\lambda.
\]
Thus, 
\[
(\iota_n\otimes\phi_\lambda)\circ(\iota_n\otimes\psi_\lambda)[x]=\displaystyle\sum_{j=1}^{k_\lambda} t_j^\lambda\otimes\phi_\lambda(g_j^\lambda).
\]
Because $\lim_\lambda\left\|x-(\iota_n\otimes\phi_\lambda)\circ(\iota_n\otimes\psi_\lambda)[x] \right\|=0$, 
for each $\varepsilon>0$ there exists a separable $y\in (C(S^1)^{(n)}\omin\A)_+$ with $\|x-y\|<\varepsilon$.
\end{proof}

The second consequence of Gurvits' separation theorem extends Theorem \ref{acp proof}, showing 
that every positive linear map $\psi$ of a unital nuclear C$^*$-algebra $\A$
is ``Toeplitz completely positive.''

\begin{corollary}\label{nuclear2} If $\A$ is a unital nuclear C$^*$-algebra and if $\psi:\A\rightarrow\B(\H)$ is a positive linear map,
then
\[
\left[ \begin{array}{cccc} 
\psi(a_0) & \psi(a_{-1}) &    \dots&  \psi(a_{-n+1}) \\
\psi(a_1) & \ddots & \ddots& \vdots\\
\vdots & \ddots & \ddots & \psi(a_{-1})  \\
 \psi(a_{n-1}) & \dots &  \psi(a_{1})&\psi(a_0)
\end{array}
\right]
\]
is a positive operator on $\displaystyle\bigoplus_1^n\H$, for every positive Toeplitz matrix
\[
\left[ \begin{array}{cccc} 
a_0 & a_{-1}   &  \dots&   a_{-n+1} \\
a_1 & \ddots & \ddots & \vdots \\
\vdots & \ddots & \ddots & a_{-1}  \\
a_{n-1} &   \dots &  a_{1}&a_0
\end{array}
\right]
\]
over $\A$.
\end{corollary}

\begin{proof} Assume that $x=[a_{k-\ell}]_{\ell,k=0}^{n-1}$ is a positive Toeplitz matrix over $\A$
and let  $\psi^{(n)}=\iota_n\otimes\psi$, where $\iota_n$ is the identity map on $C(S^1)^{(n)}$. 
By Corollary \ref{nuclear}, for each $k\in\mathbb N$  there is a positive separable Toeplitz
matrix $x_k\in C(S^1)^{(n)}\omin\A$ such that
$\|x-x_k\|<1/k$. In writing $x_k$ as 
\[
x_k=\sum_{j=1}^{m_k} t_j^{(k)}\otimes g_j^{(k)},
\]
for some $t_j^{(k)}\in(C(S^1)^{(n)})_+$ and $g_j^{(k)}\in\A_+$, we have
\[
 \psi^{(n)}(x_k ) = \sum_{j=1}^{m_k} t_j^{(k)}\otimes \psi(g_j^{(k)}),
\]
which is a positive element of $C(S^1)^{(n)}\omin\B(\H)$.
As $\psi^{(n)}$ is norm-continuous, 
$\|\psi^{(n)}(x)-\psi^{(n)}(x_k)\|\rightarrow0$ as $k\rightarrow\infty$. Because the spectrum of each $x_k$ is nonnegative, 
the upper semicontinuity of
the spectrum as a set-valued function
implies that the selfadjoint operator $\psi^{(n)}(x)$ also has nonnegative spectrum. Hence, $\psi^{(n)}(x)$
is a positive operator on $\displaystyle\bigoplus_1^n\H$.
\end{proof}
 
The maximally entangled state in quantum theory is  
the density operator
\[
\rho= \sum_{k=1}^d \vert\psi_k\rangle\langle\psi_k\vert
\]
acting on a $d$-dimensional Hilbert space $\H$, where $\langle \psi_1\vert, \dots, \langle\psi_d\vert$ is a given orthonormal basis
of $\H$.
As the bra-vectors $\vert\psi_1\rangle, \dots, \vert\psi_d\rangle$ form a dual basis of (the dual space) $\H^d$ relative to the basis of ket-vectors
$\langle \psi_1\vert, \dots, \langle\psi_d\vert$ of $\H$, it seems natural to say, for a finite-dimensional operator system $\oss$ with linear basis 
$\{s_1,\dots,s_m\}$ and dual basis $\{\delta_1,\dots,\delta_m\}$, that 
 the element 
\begin{equation}\label{e:kavruk}
\xi=\sum_{j=1}^m s_j\otimes \delta_j  
\end{equation}
is maximally entangled in $\oss\omin \oss^d$. This terminology is drawn from \cite[Appendix A]{kavruk2015}.

To apply this notion in the case where $\oss=C(S^1)^{(n)}$, we denote the dual basis of the linear basis $\{r_{-n+1},\dots,r_{n-1}\}$ of
$C(S^1)^{(n)}$ by
$\{\delta_{-n+1},\dots,\delta_{n-1}\}$. We have the identification
$C(S^1)^{(n)} \simeq \left( C(S^1)_{(n)}\right)^d$
via the unital complete order isomorphism $\phi:C(S^1)^{(n)} \rightarrow \left( C(S^1)_{(n)}\right)^d$ given in
Theorem \ref{main result proof}. The effect of $\phi$ on the canonical linear basis of $C(S^1)^{(n)}$ is
\[
\phi(r_k)=\mathfrak e_{-k}, 
\]
for every $k=-n+1,\dots,n-1$.
Note that the linear basis $\{\mathfrak e_{-n+1},\dots,\mathfrak e_{n-1}\}$ of $\left( C(S^1)_{(n)}\right)^d$ is dual to the linear
basis $\{\chi_{-n+1},\dots,\chi_{n-1}\}$ of $C(S^1)_{(n)}$. Therefore, if $\psi$ is the unital complete order isomorphism that implements
$\left(C(S^1)^{(n)}\right)^d\simeq  C(S^1)_{(n)}$, then
\[
\psi(\delta_k)=\chi_{-k},  
\]
for each $k=-n+1,\dots,n-1$. Thus, in equation (\ref{e:kavruk}) above, we replace each $\delta_k$ with $\chi_{-k}$ 
and arrive at the following definition.

\begin{definition} The element  $\xi_n\in C(S^1)^{(n)}\otimes C(S^1)_{(n)}$ defined by
\[
\xi_n= \sum_{k=-n+1}^{n-1} r_k \otimes \chi_{-k}
\]
is called the \emph{maximally entangled Toeplitz matrix}.
\end{definition}

Observe that $\xi_n$ can be view as the following function $S^1\rightarrow C(S^1)^{(n)}$:
\begin{equation}\label{e:metf}
\xi_n(z)= 
\left[ \begin{array}{cccccc} 
1 & z  & z^2 &  \dots&  z^{n-2}& z^{n-1} \\
z^{-1} & 1 & z & z^2& \dots & z^{n-2} \\
z^{-2} & z^{-1} & 1 & z &\ddots&   \vdots\\
\vdots & \ddots & \ddots & \ddots &\ddots &z^2 \\
z^{2-n}&   &  \ddots & \ddots &\ddots & z \\
z^{1-n} & z^{2-n} & \dots & z^{-2}  &z^{-1}&1
\end{array}
\right].
\end{equation}

In light of the Gurvits separability theorem, it is not obvious {\it a priori} that the maximally entangled Toeplitz matrix $\xi_n\in C(S^1)^{(n)}\omin C(S^1)_{(n)}$ is entangled
in the sense of Definition \ref{def:se}; therefore, a proof of entanglement is given in 
Corollary \ref{ent} below. Similarly, when considered as an element of the operator system $C(S^1)^{(n)}\omin C(S^1)$, the Toeplitz matrix
$\xi_n$ is a norm limit of a sequence of positive separable elements, by Corollary \ref{nuclear}; 
nevertheless, $\xi_n$ is entangled in this operator system, as shown below.

\begin{proposition}\label{ent big} The matrix $\xi_n$ is entangled in $C(S^1)^{(n)}\omin C(S^1)$, if $n\geq2$.
\end{proposition}

\begin{proof}
Assume, on the contrary, that there are nonzero positive Toeplitz matrices $t_1,\dots, t_k\in C(S^1)^{(n)}$ and nonzero positive functions
$f_1,\dots,f_k\in C(S^1)$ such that
\[
\xi_n=\sum_{j=1}^m t_j\otimes f_j.
\]
Considering these elements as matrix-valued functions, we have, for every $z\in S^1$,
\begin{equation}\label{sep eq1}
\xi_n(z)=\sum_{j=1}^m f_j(z)t_j .
\end{equation}
Fix $z\in S^1$. The positive matrix $\xi_n(z)$ has rank-1; hence, equation (\ref{sep eq1}) shows that every vector in
$\mathbb C^n$ annihilated by $\xi_n(z)$ is also annihilated by $f_j(z)t_j$ for every $j$. In passing to orthogonal complements
and making use of the fact that these matrices are selfadjoint, we deduce that the rank of each $f_j(z)t_j$ is $0$ or $1$ and that
the range of each $f_j(z)t_j$ is contained in $\xi_n(z)$. Hence, there are scalars $\alpha_j(z)\in [0,1]$, for $j=1,\dots,m$, such that
\begin{equation}\label{sep eq2}
f_j(z)t_j =\alpha_j(z)\xi_n(z).
\end{equation}
Now allowing $z$ to be arbitrary, we see that equation (\ref{sep eq2}) holds for every $z\in S^1$.

As there are no $z\in S^1$ for which $\xi_n(z)$ is diagonal, at least one $t_j$ must be nondiagonal. With such a $j$, write
$t_j=\left[ \tau^{(j)}_{k-\ell}\right]_{k,\ell=0}^{n-1}$. As $t_j$ is positive and nonzero, its diagonal entry satisfies
$\tau_0^{(j)}>0$; and because $t_j$ is not diagonal, there is
a $\ell>0$ such that $\tau_\ell^{(j)}\not=0$. Thus,
equation (\ref{sep eq2}) yields
\[
f_j(z)\tau_0^{(j)}=\alpha_j(z) \mbox{ and } f_j(z)\tau_\ell^{(j)} = \alpha_j(z)z^{-\ell},
\]
for every $z\in S^1$. This equation above shows that $\alpha_j(z)$ is continuous in $z$ and 
\begin{equation}\label{sep eq3}
\alpha_j(z)=\tau_0^{(j)}f_j(z) = \tau_\ell^{(j)} \left(f_j(z)z^\ell\right),
\end{equation}
for every $z\in S^1$. Hence, equation (\ref{sep eq3}) implies that
the Fourier coefficients of $f_j $ and $f_j\chi_\ell$ agree at every $k\in\mathbb Z$, which can happen only if $f_j$ is identically zero. Because $f_j=0$ 
contradicts the hypothesis that the functions $f_1,\dots, f_m\in C(S^1)$ be nonzero, it must be that $\xi_n$ is not separable.
\end{proof}

\begin{corollary}\label{ent} The matrix $\xi_n$ is entangled in $C(S^1)^{(n)}\omin C(S^1)_{(n)}$, if $n\geq2$.
\end{corollary}

\begin{proof} Since $C(S^1)^{(n)}\omin C(S^1)_{(n)}\subseteq_{\rm coi} C(S^1)^{(n)}\omin C(S^1)$, 
if $\xi_n$ were separable in $C(S^1)^{(n)}\omin C(S^1)_{(n)}$, then it would be separable in 
$C(S^1)^{(n)}\omin C(S^1)$ as well, in contradiction to Proposition \ref{ent big}.
\end{proof}

Besides the notion of entanglement, another property of relevance to elements of convex cones is that of purity.
 
\begin{definition} An element $x$ of a convex cone $\mathcal C$ is \emph{pure} if
the equation $x=y+z$, for $y,z\in\mathcal C$, implies that $z=\lambda x$ and $y=(1-\lambda)x$
for some real number $\lambda\in[0,1]$.
\end{definition}

Suppose that $f$ and $g$ are positive continuous functions $S^1\rightarrow C(S^1)^{(n)}$
such that $f+g=\xi_n$. Select $\lambda\in S^1$. The positive Toeplitz matrix $\xi_n(\lambda)$
has rank 1, and so it is a pure element of the cone $\M_n(\mathbb C)_+$, which implies that it is 
a pure element of the subcone $(C(S^1)^{(n)})_+$. Hence, from $f(\lambda)+g(\lambda)=\xi_n(\lambda)$,
there is a scalar $\alpha(\lambda)\in[0,1]$ such that $f(\lambda)=\alpha(\lambda)\xi_n(\lambda)$.
Thus, the mapping $z\mapsto\alpha(z)$ is
a function $\alpha:S^1\rightarrow[0,1]$ in which $\alpha(z)=\frac{\|f(z)\|}{\|\xi_n(z)\|}$ for every $z\in S^1$.
Hence, $\alpha$ is continuous, and we obtain $f(z)=\alpha(z)\xi_n(z)$ for all $z\in S^1$. 
The following proposition indicates that $\alpha$ is a constant function if the Fourier coefficients of $f$ 
vanish for all $|k|\geq n$.

\begin{proposition}\label{pure omin}
The maximally entangled Toeplitz matrix $\xi_n$ is pure in the convex
cone $\left(C(S^1)^{(n)}\omin C(S^1)_{(n)}\right)_+$.
\end{proposition}

\begin{proof}
If $\mathcal W$ is a finite-dimensional vector space, then the 
tensor product $\mathcal W\otimes \mathcal W^d$ 
is linearly isomorphic to $\mathcal L(\mathcal W)$, 
the vector space of linear transformations on $\mathcal W$. 
If we apply this linear isomorphism to a finite-dimensional 
operator system $C(S^1)^{(n)}$ and its operator system dual $C(S^1)_{(n)}$,
then the cone $\mathcal C\mathcal P(C(S^1)^{(n)})$ in $\mathcal L(C(S^1)_{(n)})$ 
of completely positive linear maps on $C(S^1)^{(n)}$ determines a cone in 
$C(S^1)^{(n)}\otimes C(S^1)_{(n)}$: namely, 
$\left(C(S^1)^{(n)}\omin C(S^1)_{(n)}\right)_+$ 
\cite[\S4]{kavruk--paulsen--todorov--tomforde2011}. 

The canonical linear isomorphism 
between $\osr\otimes\osr^d$ and $\mathcal L(\osr)$ is the one 
that maps elementary tensors $x\otimes\psi\in\osr\otimes\osr^d$ 
to rank-1 linear transformations $r\mapsto\psi(r)x$, for $r\in\osr$. Let
$\Gamma$ be the inverse of this linear isomorphism and take $\osr=C(S^1)^{(n)}$, thereby obtaining a
linear isomorphism in which
\[
\Gamma\left(\mathcal C\mathcal P(C(S^1)^{(n)})\right)=\left(C(S^1)^{(n)}\omin C(S^1)_{(n)}\right)_+.
\]
Observe that $\phi\in\mathcal C\mathcal P(\osr)$ is pure in the cone 
$\mathcal C\mathcal P(C(S^1)^{(n)})$ if and only if
$\Gamma(\phi)$ is pure in the cone 
$\left(C(S^1)^{(n)}\omin C(S^1)_{(n)}\right)_+$.

Let $\iota:C(S^1)^{(n)}\rightarrow C(S^1)^{(n)}$ be the identity map of 
$C(S^1)^{(n)}$ and note that $\Gamma(\iota)=\xi_n$. Thus, we aim to show that
$\iota$ is a pure completely positive linear map.
To this end, suppose that $\iota=\vartheta+\omega$, for two completely
positive linear maps $\vartheta,\omega:C(S^1)^{(n)}\rightarrow C(S^1)^{(n)}$. Viewing these maps
as completely positive linear maps from $C(S^1)^{(n)}$ into $\M_n(\mathbb C)$,
they admit completely positive linear extensions $\Theta,\Omega:\M_n(\mathbb C)\rightarrow\M_n(\mathbb C)$.
Hence, $\Delta=\Theta+\Omega$ is a ucp map on $\M_n(\mathbb C)$ such that, for each $x\in C(S^1)^{(n)}$,
$\Delta(x)=\theta(x)+\omega (x)=x$.
Therefore, by the Boundary Theorem \cite{arveson1972,arveson2008,farenick2011b}, $\Delta$ is the identity map of $\M_n(\mathbb C)$.
Because the identity map is an irreducible representation  of $\M_n(\mathbb C)$,
it is pure in the cone of completely positive linear maps on $\M_n(\mathbb C)$. Hence,  for some $\lambda\in[0,1]$,
$\Theta=\lambda\Delta$ and $\Omega=(1-\lambda)\Delta$, implying that
$\vartheta=\lambda\iota$ and $\omega=(1-\lambda)\iota$.
\end{proof}

 One final point of interest regarding the maximally entangled Toeplitz matrix: it is universal for all (spatially) positive Toeplitz matrices over C$^*$-algebras.

\begin{theorem}[Ando] If $\A$ is a unital C$^*$-algebra and $x\in (C(S^1)^{(n)}\omin\A)_+$, then there exists a completely positive linear map (possibly non-unital)
\[
\phi:C(S^1)^{(n)}\omin C(S^1)_{(n)}\rightarrow C(S^1)^{(n)}\omin\A
\]
such that $\phi(\xi_n)=x$.
\end{theorem}

\begin{proof} Write $x\in (C(S^1)^{(n)}\omin\A)_+$ as $x=\displaystyle\sum_{k=-n+1}^{n-1}r_k\otimes a_k$, for some $a_k\in\A$. Fix $\varepsilon>0$ and set
$y=r_0\otimes(a_0+\varepsilon1)^{-1/2}$. Let
\[
\tilde x = yxy = \sum_{k=-n+1}^{n-1}r_k\otimes b_k,
\]
where $b_0=1\in\A$ and $b_k=(a_0+\varepsilon1)^{-1/2}a_k(a_0+\varepsilon1)^{-1/2}$. By \cite[Theorem 4]{ando1970} (see also \cite[Theorem 6.5]{argerami2019}),
there is a ucp map $\psi:C(S^1)\rightarrow \A$ such that $\psi(z^k)=b_k$, for each $k$. Because the transpose map on Toeplitz matrices is 
induced by a unitary similarity transformation, we may assume that $\psi$ in fact satisfies $\psi(z^k)=b_{-k}$, for each $k$. Let $\psi_{n}$ denote the restriction of $\psi$ to the
operator subsystem $C(S^1)_{(n)}$ of $C(S^1)$ and define $\phi:C(S^1)^{(n)}\omin C(S^1)_{(n)}\rightarrow C(S^1)^{(n)}\omin\A$ by 
\[
\phi=(y^{-1})^*(\iota_n\otimes\psi_{n})y^{-1},
\]
where $\iota_n$ is the identity map on $C(S^1)^{(n)}$.
Thus, $\phi$ is completely positive and $\phi(\xi_n)=x$.
\end{proof}

\section{Conclusion}

The identification in the operator system category 
of the operator system $C(S^1)^{(n)}$ of $n\times n$ Toeplitz matrices with the operator system dual of 
the space $C(S^1)_{(n)}$ of trigonometric polynomials of degree less than $n$
has a number of striking consequences, including the implications that every positive linear map
of the Toeplitz matrices is completely positive and every unital $\M_{n}(\mathbb C)$-valued linear isometric
map of $C(S^1)^{(n)}$ is completely isometric. This identification also allows for a clearer understanding of positivity
for block Toeplitz matrices, distinguishing block Toeplitz matrices with blocks that are Toeplitz matrices from 
block Toeplitz matrices with blocks that are arbitrary complex matrices.

An operator system $\oss$ has the \emph{double commutant expectation property} if, for every unital complete order embedding 
map $\kappa:\oss\rightarrow\B(\H)$, there exists a unital completely positive linear map $\phi:\B(\H)\rightarrow\left(\kappa(\oss)\right)''$
such that $\phi\circ\kappa=\kappa$. Using the work developed in \cite{kavruk--paulsen--todorov--tomforde2013}, one can show that, for the
Toeplitz operator system $C(S^1)^{(n)}$, the double commutant expectation property is equivalent to the assertion that
$C(S^1)^{(n)}\omin\B=C(S^1)^{(n)}\omax\B$ for every unital C$^*$-algebra $\B$. (The equality 
$C(S^1)^{(n)}\omin\B=C(S^1)^{(n)}\omax\B$
is known to hold for all nuclear C$^*$-algebras $\B$ \cite{kavruk--paulsen--todorov--tomforde2011}.)
The case $n=2$ provides some insight into the general situation. To show that $C(S^1)^{(2)}$ has the double commutant expectation property, it is
sufficient to prove that every ucp map $\phi:C(S^1)^{(2)}\rightarrow \B(\H)$ has a ucp extension $\tilde\phi:\M_2(\mathbb C)\rightarrow\left(\phi(C(S^1)^{(2)})\right)''$.
By Choi's criterion for complete positivity, the extension $\tilde\phi$ must have the property that
\[
\left[ \begin{array}{cc} \tilde\phi(e_{11}) & \tilde\phi(e_{12}) \\ \tilde\phi(e_{21}) & \tilde\phi(e_{22}) \end{array}\right] =
\left[ \begin{array}{cc} \tilde\phi(e_{11}) &  \phi(r_{-1}) \\  \phi(r_1) & \tilde\phi(e_{22}) \end{array}\right]
\]
is positive. Because $1=\phi(r_0)=\tilde\phi(e_{11})+\tilde\phi(e_{22})$, the question becomes: given $\phi(r_1)=b$, then is there
a positive operator $a$ and a ucp map $\tilde\phi:\M_2(\mathbb C)\rightarrow\left(\phi(C(S^1)^{(2)})\right)''$ such that
$a=\tilde\phi(e_{11})$, $\tilde\phi(e_{22})=1-a$, and $b=\tilde\phi(r_1)$? Ando's theorem \cite{ando1973,argerami2019,arveson1972} 
gives an affirmative answer. Indeed, since the numerical radius
of the matrix $r_1$ is $1/2$, the numerical radius of $\phi(r_1)$ is at most $1/2$; hence, by Ando's theorem,
there does indeed exist a positive operator $a$ such that the map
$\tilde\phi:\M_2(\mathbb C)\rightarrow\B(\H)$ in which $\tilde\phi(e_{11})=a$, $\tilde\phi(e_{22})=1-a$, and $\tilde\phi(e_{21})=\phi(r_1)$ is completely positive.
Some additional arguments can be made (see \cite[Theorem 5.4]{farenick--kavruk--paulsen--todorov2014})
to show that $a$ is in the von Neumann algebra $\left(\phi(C(S^1)^{(2)})\right)''$, thereby 
proving that $C(S^1)^{(2)}$ has the double commutant expectation property.
For $n>2$, a rather sophisticated version of Ando's numerical radius theorem in several variables would need to hold for $C(S^1)^{(n)}$ 
to exhibit the double commutant expectation property. 

The geometry of the positive cone of $C(S^1)^{(n)}$ is analysed in considerable detail in \cite{connes-vansuijlekom2020} and further reflection upon the results in that
work may reveal previously unobserved phenomena concerning positive Toeplitz matrices. 

\section*{Acknowledgement}

I am grateful to Mizanur Rahaman for drawing my attention to the works of Gurvits \cite{gurvits--barnum2002} 
and Ando \cite{ando2004} on the separability of positive Toeplitz matrices, 
and to the referee for a careful review of the manuscript.
  

\end{document}